\documentclass[%
 aip,
 cha,
 amsmath,amssymb,
 reprint,%
]{revtex4-1}
\usepackage[utf8]{inputenc}
\usepackage[T1]{fontenc}
\usepackage{graphicx}
\usepackage{overpic}

\usepackage{longtable}
\usepackage{wrapfig}
\usepackage{rotating}
\usepackage[normalem]{ulem}
\usepackage{amsmath}
\usepackage{enumitem}
\usepackage{amssymb}
\usepackage{capt-of}
\usepackage{hyperref}
\usepackage{minted}
\usepackage{tikz}
\usetikzlibrary{arrows.meta, positioning}
\usepackage{amsmath,amssymb,bm, tikz, amsthm}
\usepackage{diffcoeff}

\usepackage{setspace, cancel}
\usepackage{parskip, xcolor}
\usepackage[most]{tcolorbox}
\usepackage[export]{adjustbox}
\renewcommand{\u}{\mathbf{u}}

\newcommand{\ds}{d} % dim of single node (System)
\newcommand{\dn}{n} % nr Nodes
\newcommand{\x}{x}% ode variable
\newcommand{\xx}{\mathbf{x}}% ode variable, bold face
\newcommand{\X}{x}% fundamental solution, normal face
\newcommand{\XX}{\mathbf{x}}% fundamental solution, normal face
\newcommand{\Lg}{{L}} %Generators of VGL algebra/tangent vector \in \Tvspace(M)
\newcommand{\vgl}{\mathfrak{v}}

\newcommand{\gl}[1]{\mathfrak{gl}(#1)}
\renewcommand{\sl}[1]{\mathfrak{sl}(#1)}
\newcommand{\aff}[1]{\mathfrak{aff}(#1)}
\newcommand{\Tvspace}{\mathfrak{X}} % space of vector fields
\newcommand{\Sf}{\bm{\Phi}} %Superposition function
\newcommand{\xixi}{\bm \xi}
\newcommand{\m} {\mathbf{X}}% microscopic variable (STACKED)
\newcommand{\mc}{\mathbf{x}}% microscopic variable component
\newcommand{\Csf}{\bm \varphi}% Collection of superposition rules
\newcommand{\M}{\mathbf{Q}}% macroscopic variable (STACKED)
\newcommand{\Mc}{{q}}% macroscopic variable component
\newcommand{\mf}{\mu} % mean-field function (scalar)
\newcommand{\mfmf}{{\bm\mu}} % mean-field function (vector)
\newcommand{\half}{\frac{1}{2}}
\renewcommand{\d}{\text{d}}
\newcommand{\N}{\mathbb{N}}

\newcommand{\R}{\mathbb{R}}
\newcommand{\C}{\mathbb{C}}

\theoremstyle{definition} % this style applies to the following
\newtheorem{definition}{Definition}

\theoremstyle{remark} % this style applies to the following

\theoremstyle{plain} % this style applies to the following

\newtheorem{lemma}{Lemma}
\newtheorem{theorem}{Theorem}

\usepackage{dcolumn}% Align table columns on decimal point
\usepackage{bm}% bold math
\usepackage[utf8]{inputenc}
\usepackage[T1]{fontenc}
\usepackage{mathptmx}
\usepackage{etoolbox}

\begin{document}

\title{Lie Meets Network Dynamics:\\ Exact Macroscopic Reductions (Finite Systems)}

\author{Erik A. Martens}
\affiliation{IMFUFA, Department of Science and Environment, Roskilde University, Roskilde, Denmark}
\affiliation{Centre for Mathematical Sciences, Lund University, M\"arkesbacken 4, 223 62 Lund, Sweden}
\email{erik.martens@math.lth.se; emartens@ruc.dk}

\author{Sanjay Dharmavaram}
\affiliation{Department of Mathematics and Statistics, Bucknell University, Lewisburg, PA 17837, USA}
\email{sd045@bucknell.edu}
\medskip
\keywords{Network dynamical systems, dimensional reduction, Lie-Scheffers theory, Kuramoto model, quaslinear equation, Riccati equation, generalized Bernoulli equation}

\begin{abstract}
We establish a unified theory for exact dimensional reductions in network dynamical systems using Lie-Scheffers theory. For network dynamical systems with \emph{mean-field Lie-Scheffers structure}, we prove that networks of $\dn$ nodes with local dimension $\ds$ can be exactly reduced from $\dn\ds$ dimensions to a fixed macroscopic system of dimension $m\ds$, where $m$ is the number of fundamental solutions required by the nodal dynamics. Crucially, the superposition rule resulting from the Lie-algebraic structure allows the mean-field coupling to be expressed explicitly in terms of the macroscopic variables, yielding a \emph{closed} self-consistent system \textit{independent of the network} size $\dn$. This reduction collapses the high-dimensional network flow onto (typically) low dimensional invariant manifolds parameterized by $\gamma=\ds(\dn-m)$ independent constants of motion.  We illustrate the theory on systems with mean-field Lie-Scheffers structure: ensembles of Riccati equations  (recovering known reductions for the Kuramoto model and Theta neuron model), recently studied quasi-linear ODEs, as well as a new class of generalized Bernoulli equations; we explicitly derive the macroscopic flows and conserved quantities for each case. Thus, our theory rigorously explains known reductions and provides a \emph{systematic method to discover new ones}.
\end{abstract}

\maketitle

% \begin{quotation}
    % Networked dynamical systems encompass a large variety of applications in nature and technology, including neuroscience, epidemiology, social sciences, engineering.
    % A paradigmatic model is the Kuramoto model~\cite{Kuramoto1975,Kuramoto1984} that describes the phase evolution of $\dn$ coupled oscillators.  The dynamics of this system was found to exhibit low-dimensional dynamics with $\dn-3$ constants of motion in a seminal study by Watanabe and Strogatz~\cite{WatanabeStrogatz1993}; later work extended this reduction to phase oscillators with sinusoidal interaction functions~\cite{Bick2018c} and identified the low-dimensional dynamics to obey symmetries underlying Möbius transformations~\cite{Marvel2009}---a structure also applies to ensembles of complex-valued Riccati equations~\cite{CestnikMartens2024}.
    % A fundamental question concerns the mathematical origin of this reduction, and whether similar approaches apply to a larger class of network dynamical systems.
    % Here, we provide a rigorous mathematical theory that explains the mechanism  behind known reductions using Lie theory~\cite{lie1893vorlesungen,carinena2011lie} and extend it to systems adhering to the mean-field Lie-Scheffers structure that we define in our work.
% \end{quotation}

\section{Introduction\label{sec:introduction}}

Network dynamical systems model the time evolution of coupled units across many fields, including neuroscience, epidemiology, engineering, and social sciences \cite{Pikovsky2001,Strogatz2003,Arenas2008}.
Addressing the high dimensionality of large networks represents a central challenge in understanding how collective dynamic behaviors emerge from local interactions.
For network systems with $\dn$ nodes each carrying a $\ds$-dimensional state, the full phase space has dimension $\dn\ds$ --- thus rendering analysis unwieldy and computation prohibitively expensive for large $\dn$.

Substantial progress has been made for particular classes of \emph{identically coupled} oscillators. For populations of Kuramoto phase oscillators with sinusoidal coupling~\cite{Kuramoto1975,strogatz2000kuramoto}, Watanabe and Strogatz~\cite{WatanabeStrogatz1993,WatanabeStrogatz1994} proved the existence of $\dn-3$ independent constants of motion, thus confining the dynamics to a three-dimensional invariant manifold. Later research traced this reduction to Möbius group symmetries and Riccati-type structure~\cite{Marvel2009,Mirollo2012}, explaining why the microscopic flow is restricted to a low-dimensional invariant manifold determined by three macroscopic variables. Subsequent extensions addressed complex-valued oscillators~\cite{CestnikMartens2024} and systems governed by higher-dimensional matrix Riccati equations~\cite{lohe2019systems}.

Beyond oscillators, recent work has identified dimensional reductions in ensembles (arrays) of quasi-linear ordinary differential equations~\cite{augustsson2025quasilinear}, indicating that the phenomenon extends beyond Riccati-type flows, while Lohe~\cite{lohe2025exact} extended reduction techniques to higher-dimensional matrix Riccati flows on spheres.
However, a unifying theoretical framework encompassing the various models and explaining \emph{when} and \emph{why} such models can be reduced has been lacking.
These observations lead to two fundamental questions:
\renewcommand{\labelenumi}{\roman{enumi})}
\begin{enumerate}
    \item What structural conditions on the nodal dynamics guarantee the existence of exact dimensional reductions?
    \item Can we devise a systematic method for discovering reducible systems beyond ad hoc derivations?
\end{enumerate}

This paper addresses these questions by embedding network dimensional reductions within the classical framework of Lie-Scheffers theory~\cite{lie1893vorlesungen,carinena2011lie}\footnote{The theory is traditionally referred to as Lie-Scheffers theory~\cite{fernandez2025hamiltonian} even though Georg Scheffers only edited Lie's work~\cite{lie1893vorlesungen}}. Our key insight is that the low-dimensional macroscopic dynamics emerges when the nodal vector field admits a superposition rule --- that is, when the microscopic flow can be reconstructed from a fixed set of $m$ fundamental solutions via a (non-)linear composition rule, known as \emph{superposition rule}\footnote{Also known as \emph{superposition principle} or \emph{superposition function}.}. We focus on \emph{mean-field type network dynamical systems}, where all nodes $i=1,\ldots,\dn$ are described by the microscopic state $\mc_i(t)\in\R^{\ds}$ and obey identical dynamics $f$ and identical mean-field coupling $\bm \mf(\xx)$:
\begin{align}
    \label{eq:microscopic_homogeneous}
    \mc'_i(t) = f\big(\mc_i, \bm \mf (\xx(t)), t\big).
\end{align}
For such systems, the Lie-Scheffers structure ensures that the entire network trajectory remains confined to an invariant manifold (that is, a leaf of a foliation of the $\dn\ds$-dimensional phase space by $m\ds$-dimensional leaves), where $m$ is the number of fundamental solutions required by the Vessiot-Guldberg Lie algebra associated with the nodal dynamics defined by $f$.
Each leaf has coordinates given by $m$ macroscopic variables in which the flow evolves autonomously; the $\gamma=\ds (\dn-m)$ independent leaf parameters are constants of motion and label the different leaves.
Since the number $m$ does not scale with network size $n$, the dimension of the flow in the macroscopic variables is very low dimensional.

We demonstrate the application of our theory through three illustrative examples ranging from classical to newly identified cases:
\begin{enumerate}
    \item Ensembles of Riccati equations (encompassing the Kuramoto model, Theta neurons, and related oscillator systems~\cite{WatanabeStrogatz1993,Marvel2009,laing2018dynamics,Bick2018c});
    \item Ensembles of quasi-linear ODEs, studied recently~\cite{augustsson2025quasilinear}; and
    \item Ensembles of generalized Bernoulli equations which, to our knowledge, are treated here for the first time.
\end{enumerate}
For each case, we explicitly construct the superposition rule, derive the closed macroscopic dynamics, and identify the conserved quantities (constants of motion).

Our contributions are threefold. First, we provide a rigorous, mathematically grounded basis for existing dimensional reduction results by identifying a \emph{mean-field Lie-Scheffers} structure as the underlying mechanism. Second, we establish a systematic methodology for testing whether novel network architectures admit exact reductions, extending beyond the specific examples discussed here. Third, we lay out a path for constructing reducible network dynamical systems based on symmetry principles. While heterogeneous systems, where nodes obey distinct (but similar) dynamics, pose different challenges for exact dimensional reductions~\cite{OttAntonsen2008}, here we clarify the structural prerequisites for identically coupled settings and suggest concrete pathways toward broader applicability.

The remainder of the paper is organized as follows. Sec.~\ref{sec:problem formulation} formulates the problem. Sec.~\ref{sec:primer} serves as a primer on Lie-Scheffers theory, focusing on single-node systems. Sec.~\ref{sec:main_result} presents our main theorem and the framework for the exact dimensional reduction of network dynamical systems with a \emph{mean-field Lie-Scheffers structure}, which we define. Sec.~\ref{sec:applications} applies our framework to mean-field type network dynamical systems based on ensembles of Riccati, quasi-linear, and Bernoulli ODEs. Finally, Sec.~\ref{sec:discussion} concludes by discussing the implications of this framework and outlines paths for future research directions.

\section{Problem formulation: Mean field-type network dynamical systems\label{sec:problem formulation}}

We introduce the concept of a network dynamical system. Let the dynamic state of each node $i\in V := \{1,\ldots,\dn\}$ be described by the vector $\mc_i(t)\in\R^\ds$, and let the vector $\m(t)=(\mc_1,\ldots,\mc_\dn)^T\in\R^{\ds\dn}$ collect the states of all nodes, referred to as the \emph{microscopic state} of the network, see Fig.~\ref{fig:networkdynamics} (middle).   The dynamics of $\m$ are governed by the following initial value problem,
\begin{subequations}
\label{eq:microscopic_system}
\begin{equation}
    \m' = \mathbf{F}(\mathbf{\m},t),
    \label{eq:microscopic_system_de}
\end{equation}
subject to the initial conditions
\begin{equation}
    \m(0) = \m_o:=(\mc^o_1,\ldots, \mc^o_n)^T,
    \label{eq:microscopic_system_ic}
\end{equation}
\end{subequations}
which we refer to as the \emph{microscopic (network) system}.
We require that $\mathbf{F}:\mathbb{R}^{\ds\dn}\times\mathbb{R}\to\mathbb{R}^{\ds\dn}$ is a sufficiently smooth function such that the assumptions of the Picard-Lindelöf theorem hold. Thus, there exists a locally \emph{unique} solution to the initial value problem \eqref{eq:microscopic_system}.

We assume that the system \eqref{eq:microscopic_system} has a special structure:
\begin{definition}[Mean-field type dynamics]\label{def:meanfieldtype_dynamics}
A system of differential equations \eqref{eq:microscopic_system_de}, is said to have \emph{mean-field-type dynamics} if there exists a (sufficiently smooth) function $f:\mathbb{R}^\ds\times\mathbb{R}^p\times\mathbb{R}\to\mathbb{R}^\ds$,  and a \emph{mean-field function} $\mfmf:\mathbb{R}^{\ds\dn}\to\mathbb{R}^p$ (for some $p\in\N$) such that

\begin{equation}
    \mathbf{F}(\m,t) =\begin{bmatrix}
        f(\mc_1,\mfmf(\m),t)\\
        f(\mc_2,\mfmf(\m),t)\\
        \vdots\\
        f(\mc_n,\mfmf(\m),t)\\
    \end{bmatrix}.
    \label{eq:special structure}
\end{equation}

We refer to equations~\eqref{eq:microscopic_system} with mean-field-type dynamics~\eqref{eq:special structure} as \emph{microscopic equations} or \emph{microscopic system with mean-field-type dynamics}, since they describe the dynamics of individual nodes $i=1,\ldots,\dn$,
\begin{align}\label{eq:microscopic_system2}
    \mc'_i=f(\mc_i,\mfmf(\m),t),\quad i \in V.
\end{align}
\end{definition}

The mean-field $\mfmf(\m)$ aggregates information from (a subset of) nodal states and feeds back \emph{identically to every node}; cf.\ Fig.~\ref{fig:networkdynamics} (left). We emphasize that $\mfmf$ need not depend on all nodes---it may be computed from any subset of node states---but it must be the \emph{same} function for every node.
Admissible structures for $\mfmf$ are detailed in Sec.~\ref{sec:compatible_mean_field_functions}.

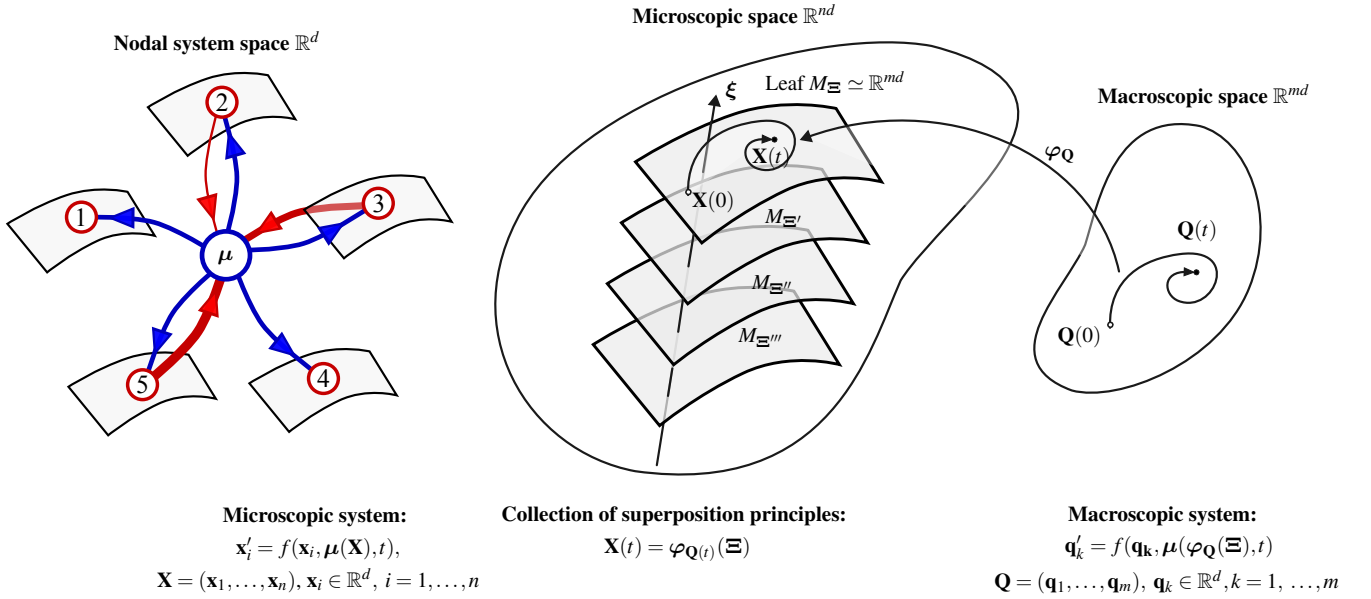
\begin{figure*}[htp!]
\centering
\begin{overpic}[width=\textwidth]{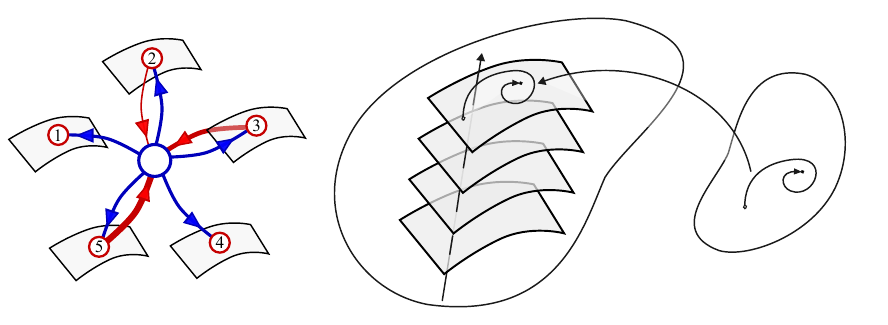}
    \footnotesize % changes font size globally in overpic scope
% % ---- left: ----
    \put(16.8,18.2){$\mfmf$}
    \put(4,35){\bf (a) Network with mean-field type dynamics}
    \put(22,30){ Nodal system space $\R^{\ds}$}
    \put(1,-4){
    \begin{minipage}{20em}
        {\bf Microscopic system:}  \vspace{0.3em}
        \\
        $\mc_i' = f(\xx_i,\mfmf(\mathbf{\m}),t),$
        \\ \smallskip
        $\mathbf{\m}=(\mc_1,\ldots,\mc_\dn)^T$,
        $\mc_i\in\R^\ds,\; i=1, \ldots, \dn$\
    \end{minipage}
    }
% ---- middle: ----
%     \put(37,35){(b)}
    \put(48,36){\bf (b) Microscopic space $\R^{\dn \ds}$}
    \put(58,31){\footnotesize Leaf $M_{\bm\Xi}\simeq\R^{m\ds}$}
    \put(58,21){\footnotesize  $M_{\bm\Xi'}$}
    \put(57,16){\footnotesize  $M_{\bm\Xi''}$}
    \put(56,12){\footnotesize  $M_{\bm\Xi'''}$}
    \put(55,30.5){$\xixi$}

    \put(52.5,22.2){\footnotesize $\m(0)$}
    \put(57,25.5){\footnotesize $\m(t)$}

    \put(79,26){$\Csf_\M$}
    \put(35,-4.7){
    \begin{minipage}{20em}
        {\bf Collection of superposition rules:}  \vspace{0.3em}
        \\
        $\m(t)=\Csf_{\M(t)}(\bm\Xi)$
        \\
        $\Csf(\M,\bm\Xi):=(\Sf(\M,\xixi_1),\ldots,\Sf(\M,\xixi_n))^T$\\
        \vspace{\baselineskip}
    \end{minipage}
    }
% ----- right: ----
%     \put(79,35){(c)}
    \put(82,30){\bf (c) Macroscopic space $\R^{m\ds}$}
    \put(80,12){\footnotesize $\M(0)$}
    \put(89,20){\footnotesize $\M(t)$}
    \put(68,-4){
    \begin{minipage}{20em}
    {\bf Macroscopic system: } \vspace{0.3em}
        \\
        $\mathbf{\Mc}_k' =     f(\mathbf{\Mc_k},\mfmf(\Csf_\mathbf{\M}(\bm\Xi),t)$
        \\ \smallskip
        $\M = (\mathbf{\Mc}_1,\ldots, \mathbf{\Mc}_m),\;\mathbf{\Mc}_k\in\R^\ds, k =   1, \;\ldots, m $ \

    \end{minipage}
    }
\end{overpic}
\bigskip\bigskip
\caption{Network dynamics for $\dn$ nodes with a mean-field forcing $\mfmf(\m)$ which receives input from a subset of nodes (red arrows) with varying weight (arrow thickness) and broadcasts to all nodes (blue arrows) in the network (panel (a)).
The microscopic space (panel (c)) is foliated with leaves labeled by $\bm \Xi$. The dynamics in the microscopic space are confined to a leaf $\bm\Xi$ due to the existence of the time-independent Vessiot-Guldberg Lie algebra/superposition rule $\Sf$ and is described by the collection of superposition rules $\bm \varphi$, parameterized by the fundamental solutions $ \M(t)$ evolving in the low dimensional macroscopic space (panel (c)).
\label{fig:networkdynamics}}
\end{figure*}

\section{Primer on Lie-Scheffers theory\label{sec:primer}}

We briefly review results of Lie-Scheffers theory required to address the problem. No originality is claimed in the content of this section --- an excellent review can be found in Ref~\cite{carinena2011lie}.  The goal of Lie-Scheffers theory is to characterize the necessary and sufficient conditions under which a system of first-order differential equations of the form
\begin{equation}
    \x'_{c}(t) = F_c(\xx,t),\quad c=1,\ldots,\ds,
    \label{eq:ode system}
\end{equation}
has a \emph{superposition rule}. For this section, we assume that the system of differential equations is defined on vector spaces, i.e., $\xx(t):=(x_1,\ldots,x_\ds)\in \mathbb{R}^\ds$ and the functions $F_c: \R^\ds\times\mathbb{R}\to\R^\ds$ that define the flow are real-valued and smooth. Note that after this section, we shall return to the notation for the system of differential equations presented in Sec.~\ref{sec:problem formulation}; in this vein, we will consider $\ds$-dimensional dynamics---encoded by $\xx_k(t)$---to occur on every node $k$ of a network.

For linear differential equations, the reader may be familiar with what is meant by the term \emph{superposition}. For example, if $\mathbf{F}(\xx,t) = A(t)\xx$, where $A(t)$ is a $\ds\times \ds$ matrix, then by the theory of systems of linear differential equations, a \emph{general solution} $\XX^{(0)}$ can be written\footnote{We use the notation of \cite{carinena2007superposition}.} as
\begin{equation}
    \XX^{(0)}(t) = \sum_{\ell=1}^\ds\xi_\ell\,\XX^{(\ell)}(t),
\end{equation}
where $\XX^{(\ell)}(t), \ell=1,\ldots,\ds$ are particular solutions of the system. Fundamental solutions are linearly independent and chosen so that they satisfy specific initial conditions. Lie-Scheffers theory \cite{lie1893vorlesungen} extends the idea of the superposition rule to nonlinear systems. For nonlinear systems, the term is defined as follows \cite{carinena2011lie}:

\begin{definition}[superposition rule and Lie System]\label{def:superposition_principle}
The system of first-order differential equations \eqref{eq:ode system} is said to admit a \emph{superposition rule}, if there exists a $t$-independent function  $\Sf:(\mathbb{R}^\ds)^m\times\mathbb{R}^\ds\to\mathbb{R}^\ds$ such that any (general) solution  of \eqref{eq:ode system} can be written in terms of any \emph{generic} set of (\textit{fundamental}) solutions  $\XX^{(1)}(t),\ldots,\XX^{(m)}(t)$ of \eqref{eq:ode system} as,
\begin{equation}
    \XX^{(0)}(t) = \Sf(\XX^{(1)}(t),\ldots,\XX^{(m)}(t),\bm{\xi}),
    \label{eq:superposition}
\end{equation}
where $\bm{\xi}\in\mathbb{R}^\ds$ is a constant vector (set by initial conditions chosen for $\XX^{(0)}(t)$). Here, \emph{generic} means that there exists an open dense subset $U \subset (\mathbb{R}^\ds)^m$, such that the expression \eqref{eq:superposition} holds for \emph{every set} of solutions $\XX^{(1)}(t),...,\XX^{(m)}(t)$ with the initial conditions $(\XX^{(1)}(0),...,\XX^{(m)}(0))\in U$.

A system of differential equations \eqref{eq:ode system} that admits a superposition rule is called a \emph{Lie system}.
\end{definition}

Note that by definition (since \eqref{eq:superposition} holds for any generic particular solution), $\Sf$ must be invariant under permutation of the first $m$ entries and corresponding permutation of entries in $\bm{\xi}$. For instance,  this can readily be seen in the case of linear systems,
\begin{equation}
    \Sf(\XX^{(1)},\ldots, \XX^{(\ds)},\bm\xi) := \sum_{\ell=1}^\ds \xi_\ell\,\XX^{(\ell)}.
\end{equation}
Thus, permuting ${\XX^{(\ell)}}$ including the corresponding $\xi_\ell$ does not change $\Sf$.

The necessary and sufficient conditions for the existence of a superposition are given by the following theorem:

\begin{theorem}[Lie-Scheffers Theorem]
\label{thm:Lie theorem}
A first-order system \eqref{eq:ode system} admits a superposition rule (i.e., it is a Lie system) if and only if the vector field $\Lg(\xx,t):= \sum_{k=1}^\ds F_k(\xx,t)\partial_{x_k}$ (associated with the integral curves of the system) can be written as
\begin{equation}
    \Lg(\xx,t) = \sum_{\alpha=1}^r a_\alpha(t)\Lg_\alpha(\xx),
    \label{eq:decomposition}
\end{equation}
for some smooth functions $a_1(t),\ldots,a_r(t)$, and $\Lg_1,\ldots, \Lg_r$ are vector fields in $\mathbb{R}^\ds$ that span an $r$-dimensional real Lie algebra of vector fields. That is,
\begin{equation}
    [\Lg_\alpha,\Lg_\beta] = \sum_{\gamma=1}^r c^\gamma_{\alpha\beta} \,\Lg_\gamma,
    \label{eq:lie:involution}
\end{equation}
for all $\alpha, \beta=1,\ldots,r$, and for some constants $c^\gamma_{\alpha\beta}\in\mathbb{R}$, and where $r<\infty$ and $[A,B]:=AB-BA$ is the \emph{commutator}  or \emph{Lie bracket}.
\end{theorem}

\noindent Note that $\Lg_\alpha$ is independent of $t$ for all $\alpha=1,\ldots,r$. The Lie algebra generated by vector fields $\Lg_\alpha$ is called the \emph{Vessiot-Guldberg Lie (VGL)} algebra $\vgl:=\text{span}\{\Lg_1(\xx),\ldots,\Lg_r(\xx)\}\subset \Tvspace(M)$ associated with the Lie-system \eqref{eq:ode system}, where $\Tvspace(M)$ is the infinite dimensional Lie algebra of all smooth vector fields on the manifold (phase space) $M=\R^\ds$. As we discuss later, the (finite-dimensional) VGL algebra can be represented via an isomorphism to a subalgebra of $\gl{\ds,\R}$. We discuss a concrete example in Sec.~\ref{sec:examples_lie_scheffers_theory} .  Example systems for which the tangent vector fields are generated by VGL algebras include the inhomogeneous first-order linear equations ($\aff{1,\R}$), homogeneous second-order linear equations ($\gl{2,\R}$), and the Riccati equations ($\sl{2,\R}$). We later provide detailed examples in Sec. \ref{sec:examples_lie_scheffers_theory} and in Sec.~\ref{sec:applications}.

Although Lie and Scheffers discussed this theorem in the late nineteenth century \cite{lie1893vorlesungen}, a rigorous proof using geometric methods was developed more recently by Bl{\'a}zquez-Sanz and Morales-Ruiz~\cite{blazquez2010local} and Cari\~{n}ena, Grabowski and Marmo \cite{carinena2007superposition, carinena2011lie}. At the core of the proof lie two ideas: \emph{diagonal prolongation} of vector fields and Frobenius' theorem \cite{lang2012differential} on the integrability of a distribution of vector fields (i.e., a subbundle of the tangent bundle).

Briefly, one extends the phase space $M$ to $M^{m+1}\cong\mathbb{R}^{\ds(m+1)}$ by making $m+1$ copies of the system (diagonal prolongation). Since the prolonged vector fields still form a Lie algebra and hence satisfy an analogous involution condition \eqref{eq:lie:involution}, Frobenius' theorem guarantees the existence of $N-r$ independent first integrals (constants of motion) on this extended space, where $N=\ds(m+1)$. These integrals define a foliation of the extended space; the leaves of the foliation correspond to trajectories sharing the same constants $\bm{\xi}$. Inverting these relations yields the superposition function $\Sf$.

\medskip
\noindent This construction allows us to formalize the foliation structure underlying the superposition rule. A critical requirement emerging from this geometric setup is determining the number of fundamental solutions, $m$, needed to uniquely reconstruct the general solution. Specifically, since the superposition rule relies on inverting the relations defined by the foliation, we must ensure there are enough first integrals to solve for $\XX^{(0)}$. This leads to the following theorem:

\begin{theorem}[Foliation of the Extended Phase Space]
\label{thm:foliation}
    If the system \eqref{eq:ode system} admits a decomposition of the form given in \eqref{eq:decomposition} in terms of the operators $\Lg_\alpha$ which form a VGL algebra (i.e., satisfy \eqref{eq:lie:involution}), then there exists a positive integer $m$ such that $\mathbb{R}^{\ds(m+1)}$ has an $r$-dimensional foliation. That is, there exists $\Psi_j:\mathbb{R}^{\ds(m+1)}\to \mathbb{R}$ ($j=1,\ldots, \ds(m+1)-r$) satisfying the PDEs
    \begin{equation}
        \hat{\Lg}_\alpha[\Psi_j]=0,
        \label{eq:Psi's pde}
    \end{equation}
    with $\alpha=1,\ldots, r$ and $j=1,\ldots ,\ds(m+1)-r$, and where $\hat{\Lg}_\alpha$ is the diagonal prolongation of $\Lg_\alpha$, defined by $\hat{\Lg}_\alpha:=\sum_{i=0}^m \Lg_\alpha(\XX^{(i)})$ ($\alpha=1,\ldots,r)$. The leaves of the foliation are defined by the conditions,
    \begin{equation}
        \Psi_j (\XX^{(0)},\ldots,\XX^{(m)})=\xi_j,
        \label{eq:inverse sup princip}
    \end{equation}
    with $j=1,\ldots,\ds(m+1)-r$ and where $\xi_j$ are real constants, called \emph{leaf parameters}. To ensure that $\Sf$ exists, $m$ and $\XX^{(1)}, \ldots, \XX^{(m)}$ are chosen such that $\hat{\Lg}_\alpha$ ($\alpha=1,\ldots,r)$ are linearly independent.
\end{theorem}

\noindent Note that when solving ordinary differential equations, $\XX^{(0)}$ plays the role of a \emph{general solution}. The theorem above establishes the existence of the first integrals $\Psi_j$ that define the foliation. However, to find the explicit superposition rule, we must be able to invert these relations to solve for $\XX^{(0)}$ in terms of the other particular solutions and the leaf parameters. The next theorem provides conditions under which this inversion is possible.

\begin{theorem}[Counting]
\label{thm:r<=mn}
    If $r\leq m\ds$ and the assumptions of the previous theorem hold, then there exists a superposition rule such that
    $$\xx^{(0)}=\Sf(\xx^{(1)},\ldots,\xx^{(m)},\bm{\xi}).$$
\end{theorem}
\noindent If $r\leq m\ds$, then $(m+1)\ds-r\geq \ds$.  Therefore, by comparison with the indexation of \eqref{eq:inverse sup princip}, we have \emph{at least} $\ds$ independent first integrals $\Psi_j$ (with $j=1,\ldots, (m+1)\ds-r$) in the sense of Frobenius' theorem. This ensures that the Jacobian of the map from the $(m+1)$ solutions to the integration constants is invertible with respect to the components of $\xx^{(0)}$. Specifically, the columns of $\nabla_{\hat\xx}\Psi=(\partial_{\hat x_i}\Psi_j)$ are linearly independent, where $\hat \xx :=(x_i^{(k)})\in\R^{(m+1)\ds}$ is the flattened vector of the $k=0,\ldots, m$ fundamental solutions. In particular, it follows that the sub-matrix $(\partial_{x^{(0)}_i}\Psi_j)$ (for $i,j=1,\ldots, \ds$) is invertible. By the Implicit Function Theorem, we can then invert \eqref{eq:inverse sup princip} to recover the superposition rule:
$$\XX^{(0)} = \Sf(\xx^{(1)},\ldots,\xx^{(m)},\bm{\xi}).$$
The converse of Theorem \ref{thm:foliation} is also true.  That is, existence of the superposition rule implies that the vector field associated with the integral curves of \eqref{eq:ode system} admits the decomposition \eqref{eq:decomposition}. For further details on the construction of $\Sf$ and the determination of minimal $m$, see Ref~\cite{carinena2007superposition}. Sec.~\ref{sec:examples_lie_scheffers_theory} outlines a concrete application of this procedure.

\begin{figure}[htp!]
    \centering
    \includegraphics[width=0.45\columnwidth]{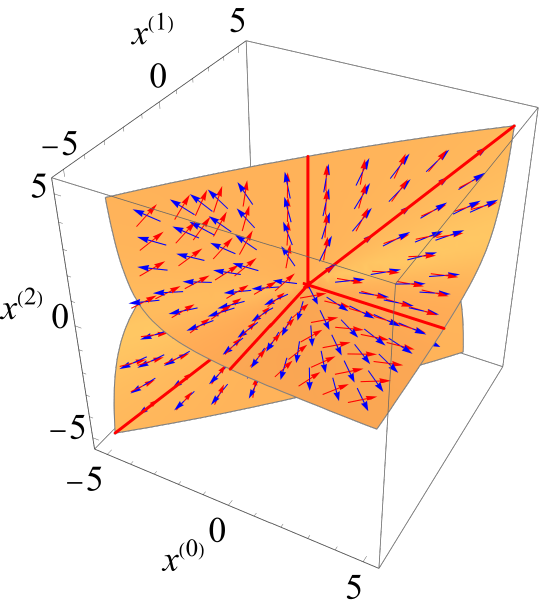}
    \includegraphics[width=0.45\columnwidth]{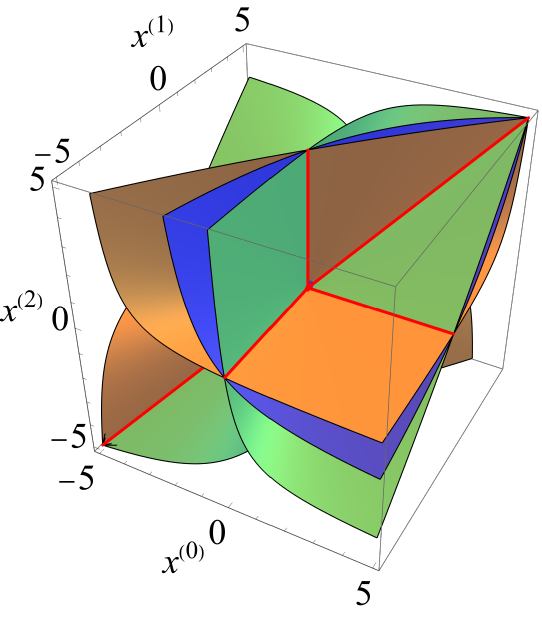}
    \caption{\label{fig:foliations}Illustration of leaves of the foliation for
    the homogeneous ($a_0\equiv0$) Riccati equation~\eqref{eq:riccati} with dynamics governed by a $r=2$-dimensional algebra. Single leaves are shown together with arrows representing the basis vector fields \eqref{eq:lin first order prolong1}-\eqref{eq:lin first order prolong2},  spanning the tangent plane (left panel), alongside multiple leaves with distinct values for $\xi$ (right panel). The tangency of the vector field to the leaf is a consequence of the involutivity condition implying integrability. The superposition rule breaks down at \emph{non-generic points} (highlighted red) where tangent vectors are collinear. }
\end{figure}

\subsection{Example: Riccati equation\label{sec:examples_lie_scheffers_theory}}

The Riccati equation for $x(t)\in\R$ ($d=1$) has the form
\begin{equation}
    \x'(t) = a_0(t) + a_1(t) \x + a_2(t)\x^2,
    \label{eq:riccati}
\end{equation}
where $a_0(t), a_1(t), a_2(t)\in\R$ are given smooth functions in time. The tangent vector field to the flow can be written as the linear combination
\begin{equation}
    \Lg = a_0(t) \Lg_1 + a_1(t) \Lg_2 + a_2(t) \Lg_3,
\end{equation}
with $\Lg_1=\partial_\x$, $\Lg_2 = \x \partial_\x$, $\Lg_3=\x^2\partial_\x$. These form an $r=3$-dimensional VGL algebra, $\vgl:=\text{span}\{\Lg_1,\Lg_2,\Lg_3\}$, closed under the Lie bracket:
\begin{equation}
\label{eq:ex_riccati_lie_bracket_vfield}
[\Lg_1,\Lg_2] =\Lg_1 ,\quad
[\Lg_1,\Lg_3]=2\Lg_2,\quad
[\Lg_2,\Lg_3]=\Lg_3.
\end{equation}
Thus, by Theorem \ref{thm:Lie theorem}, the Riccati equation \eqref{eq:riccati} is a Lie system.

The VGL algebra $\vgl$ is isomorphic to $\sl{2,\R}$. To see this, let $\sl{2,\R}$ have the standard basis
\begin{equation}\label{eq:ex_riccati_lie_bracket_sl2}
E=
\begin{pmatrix}
    0& 1 \\ 0 & 0
\end{pmatrix}
,\quad
F=
\begin{pmatrix}
    0& 0 \\ 1 & 0
\end{pmatrix}
,\quad
H=
\begin{pmatrix}
    1& 0 \\ 0 & -1
\end{pmatrix},
\end{equation}
with $[H,E]=2E$, $[H,F]=-2F$, $[E,F]=H$. Comparing bracket relations \eqref{eq:ex_riccati_lie_bracket_vfield} and \eqref{eq:ex_riccati_lie_bracket_sl2} identifies the isomorphism $\phi:\vgl\to\sl{2,\R}$ defined by
\[\phi(\Lg_1)=E,\quad \phi(\Lg_2)=\half H,\quad \phi(\Lg_3)=-F,\]
which can be verified to satisfy $\phi([\Lg_l,\Lg_k])=[\phi(\Lg_l),\phi(\Lg_k)]$. Hence $\vgl\simeq\sl{2,\R}$.

By the Counting Theorem~\ref{thm:r<=mn}, we have $m \geq r/\ds = 3$, which implies that the number of fundamental solutions is  at least $m\geq 3$. To verify that $m=3$
suffices, we seek a non-trivial first integral on the prolonged (four-fold) extended space
$(\x^{(0)}, \X^{(1)}, \X^{(2)}, \X^{(3)}) \in M^4$ (corresponding to one solution and three fundamental solutions).
From Theorem~\ref{thm:foliation}, we seek a function $\Psi$ satisfying the following simultaneous PDEs (i.e, the first integral $\Psi$ must annihilate each prolonged field). The diagonal prolongations $\hat{\Lg}_\alpha$ of the three generators $\Lg_\alpha$ yield the simultaneous PDEs:
\begin{equation}
   \sum_{i=0}^3 \partial_{\X^{(i)}} \Psi=0, \;\;
   \sum_{i=0}^3 \X^{(i)} \partial_{\X^{(i)}} \Psi=0, \;\;
   \sum_{i=0}^3 (\X^{(i)})^2 \partial_{\X^{(i)}} \Psi=0.
      \label{eq:riccati fundamental sol condition}
\end{equation}
The above first-order PDEs can be solved using the method of characteristics to obtain $\Psi$. Alternatively, it is readily verified by direct substitution that the following function,
\begin{equation}
\Psi(\x^{(0)},\X^{(1)},\X^{(2)},\X^{(3)})=\frac{(\X^{(2)}-\X^{(0)})(\X^{(1)}-\X^{(3)})}{(\X^{(3)}-\X^{(0)})(\X^{(1)}-\X^{(2)})}=\xi.
\label{eq:inv super ricatti}
\end{equation}
solves the PDEs~\eqref{eq:riccati fundamental sol condition}, that is, $\hat \Lg_\alpha[\Psi]\equiv 0$.
Note that $\xi$ represents a \emph{constant of motion} and coincides with the cross-ratio, the invariant associated with the Möbius transformation~\cite{ahlfors1979complex,knorrer2013geometrie,Marvel2009}.
More specifically, the Lie algebra $\mathfrak{sl}(2,\R)$ integrates to the Möbius group of fractional linear transformations acting on the real projective line $\mathbb{RP}^1$, under which the cross-ratio is invariant. We note that in the context of mean-field type dynamics on networks with $\dn $ nodes (see Sec.~\ref{subsec:riccati_examples})
there are $\dn-3$ functionally independent cross-ratios, corresponding to $\dn-3$ constants of motion~\cite{goebel1995comment,Marvel2009}.

Inverting \eqref{eq:inv super ricatti} yields the superposition rule:
\begin{align}
\begin{split}
    \x^{(0)}(t)&=\frac{\X^{(2)}(\X^{(1)}-\X^{(3)})-\xi \X^{(3)}(\X^{(1)}-\X^{(2)})}{(\X^{(1)}-\X^{(3)})-\xi(\X^{(1)}-\X^{(2)})}\\
    &=:\Sf(\X^{(1)},\X^{(2)},\X^{(3)},\xi).
    \label{eq:gen sol riccati}
    \end{split}
\end{align}
We observe that $\Psi$ actually can be inverted to solve for $\x^{(0)}$. This confirms that $m=3$ is a sufficient number of fundamental solutions.

We next consider the special case of the homogeneous Riccati equation ($a_0=0$),
\begin{equation}\label{eq:riccati_homog}
    \x'(t) = a_1(t) \x + a_2(t) \x^2.
\end{equation}
This special case simplifies the algebra while preserving the geometric structure, making it suitable for visualizing the foliation and dynamics on the leaves  (Fig.~\ref{fig:foliations}). We have $\Lg = a_1(t) \Lg_2 + a_2(t) \Lg_3$ with $[\Lg_2,\Lg_3]=\Lg_3$, so $r=2$. Clearly, $\X^{(3)}(t)\equiv 0$ is a solution. Substituting into \eqref{eq:inv super ricatti} and inverting, the general solution can be written with $m=2$ fundamental solutions as
\begin{equation}
    \X^{(0)}(t)=\frac{\X^{(2)}\X^{(1)}}{\X^{(1)}-\xi(\X^{(1)}- \X^{(2)})}.
    \label{eq:gen sol special riccati}
\end{equation}
The inverse superposition rule is then
\begin{equation}
\Psi(\X^{(0)},\X^{(1)},\X^{(2)}) = \frac{(\X^{(0)}-\X^{(2)}) \X^{(1)}}{(\X^{(1)}-\X^{(2)})\X^{(0)}}=\xi.
\end{equation}
In light of Theorem \ref{thm:foliation}, $\Psi$ is a first integral of the diagonally prolonged vector fields
\begin{align}
\label{eq:lin first order prolong1}
    \hat{\Lg}_1&= \X^{(0)}\partial_{\X^{(0)}} + \X^{(1)}\partial_{\X^{(1)}} + \X^{(2)}\partial_{\X^{(2)}},
    \\
    \hat{\Lg}_2&=\left({\X^{(0)}}\right)^2\partial_{\X^{(0)}} + \left({\X^{(1)}}\right)^2\partial_{\X^{(1)}} + \left({\X^{(2)}}\right)^2\partial_{\X^{(2)}},
    \label{eq:lin first order prolong2}
\end{align}
that is, $\hat{\Lg}_1(\Psi)=0$ and $\hat{\Lg}_2(\Psi)=0$ (cf.~\eqref{eq:Psi's pde}), which can be verified by direct computation. Since the prolonged fields are three-dimensional, their relationship with the foliation can be visualized; see Fig.~\ref{fig:foliations}.

\subsubsection{The complex-valued Riccati equation}
Our framework is formulated over the real number field; but the case of the Riccati equation extends naturally to the complex Riccati equation $\dot{z} = a_0(t) + a_1(t)z + a_2(t)z^2$ with $z \in \mathbb{C}$ and $a_0,a_1,a_2\in\C$. The VGL algebra becomes $\mathfrak{sl}(2,\mathbb{C})$ acting on $\mathbb{CP}^1$ via Möbius transformations, with identical Lie bracket relations, superposition rule, and complex-valued cross-ratio~\cite{blazquezsanz2012}.

\subsubsection{Group-theoretic construction of superposition rules}
There is an elegant alternative to the Frobenius-based approach to construct the superposition rule using group-theoretic methods~\cite{anderson1982systems}.
We briefly outline the procedure. First, given a Lie system defined on the phase space $M$ with a Vessiot-Guldberg algebra $\vgl$, one chooses a matrix representation $\rho:\vgl \to \gl{V}$ on a vector space $V$. Thus, each generator $\Lg_\alpha$ is associated with a constant matrix $\rho(\Lg_\alpha)$ and the matrix Lie group $G\subset \text{GL}(V)$ is generated by the matrix algebra $\rho(\vgl)$; the linear Lie algebra action exponentiates to a linear group action $g\in G$, so that we write $g\cdot \mathbf z = \rho(g)\mathbf z$ for $g\in G,\mathbf z \in V $.

 We aim at replacing the nonlinear ODE on $M$ with a linear ODE on $V$, with solutions $x(t)\in M$ and $\mathbf z(t) \in V$. To relate the dynamics in the (generally) nonlinear space $M$ and the linear space $V$, one introduces a surjective map $\pi:U\subset V\to M$ with $\mathbf z(t)\mapsto x(t) = \pi(\mathbf z (t)) $. Equivariance ensures consistency between the dynamics in $V$ and $M$: $\pi(g\cdot \mathbf z)=g\circ \pi(\mathbf z)$ for all $g\in G, \mathbf z \in U$, that is, the action on the left is the linear action on $V$ and the action on the right is the induced (non-)linear action on $M$. The dynamics on $V$ evolves according to $ \dot {\mathbf z} (t) = A(t) \mathbf z(t)$,
where $A(t)=\sum_{\alpha=1}^r a_\alpha(t)\rho(\Lg_\alpha)\in\rho(\vgl)$. The solution is expressed via the fundamental matrix $g$, i.e. $\mathbf z(t) =g(t)\mathbf{z}_0$; differentiating, we obtain the propagator equation:  $\dot g(t)=A(t) g(t)$ with $g(0)=\mathbf I$.
Solving for $g(t)$ integrates the infinitesimal Lie algebra elements $\rho(\Lg_\alpha)$ to obtain the finite Lie group (matrix) element $g(t)\in G$.  The general solution on $M$ is then recovered via $x(t)=\pi(g(t) \mathbf z_0)$. This equivariant map converts the linear group action of $G$ on $V$ into a nonlinear action $G \curvearrowright M$; the trajectory remains on the group orbit $G\circ x_0$.
Finally, we would like to express the superposition rule $\Sf$ as a map parameterized by a finite set of  fundamental solutions. However, the matrix entries of $g(t)$ are in general unknown. To eliminate these entries, we track a finite set of known trajectories, corresponding to these fundamental solutions, $\x^{(1)}(t),\x^{(2)}(t),\x^{(3)}(t)$, which yields a system of linear algebraic equations for the unknown group components.  By inverting this linear system, we can explicitly solve for and eliminate the group components in favor of the fundamental solutions.

This perspective reveals the structural relationship between the linear flow given by the propagator equation on the higher-dimensional space $V$ and the (non-)linear flow on the phase space $M$. The propagator equation bridges the abstract Lie algebra structure and the concrete superposition rule.

\medskip \noindent \textit{Remark.} The structural relationship between $V$ and $M$ given by $\pi$  depends on the specific Lie system given.
For Riccati equations, $\pi(z_1,z_2) = z_1/z_2$ is a projection from $V=\R^2$ to $M=\mathbb{RP}^1$, collapsing the radial direction. For quasi-linear systems, the group already acts linearly on $M=\R^d$ and $\pi$ plays the role of  identity ($V=M=\R^d$). For generalized Bernoulli equations, $x=\pi(y) = y^{1/(1-p)}$ is a diffeomorphism between $V=\R$ and $M=\R$, pulling back the linear $\aff{1}$ action to nonlinear dynamics on the phase space $M=\R$. See Sec.~\ref{sec:applications} for explicit constructions in these cases.

\medskip \noindent \textit{Remark.} The dimension of the VGL algebra, $r=\dim(\vgl)=\dim(G)$, determines the minimum number of fundamental solutions via $m\geq r/\ds$. The resulting macroscopic space has dimension $m\ds\geq r$, and the group orbit of the group action on $M$ have dimension $r$. When $r=m\ds$, the invariant leaves coincide with the group orbits.

\subsection{Lie systems vs.\ Lie symmetries}

Before proceeding, we clarify a potential confusion in terminology. \emph{Lie systems} and \emph{Lie point symmetries of differential equations} both involve Lie algebras of vector fields, but they address fundamentally different questions.
Lie point symmetries map solutions of an ODE to other solutions of the same ODE. They are commonly used to reduce the order or find first integrals~\cite{olver1993applications,hydon2000symmetry,ibragimov1992group}. Lie-Scheffers systems, on the other hand, are ODEs whose time-dependent vector field lies within a finite-dimensional  Vessiot-Guldberg Lie algebra $\vgl$. These structures generate the dynamics via the superposition rule~\cite{carinena2011lie,blazquez2010local}, provided that the conditions in Theorem~\ref{thm:Lie theorem} are satisfied.
These structures are generally distinct: a symmetry algebra of the equation does not need to coincide with the algebra generating the flow. In particular, if an ODE satisfies a Lie point symmetry, the existence of a superposition rule is not implied.

\section{Main Result: Lie-Scheffers Theory of mean-field-type network dynamical systems\label{sec:main_result}}
In this section, we present the main results of this work. We show that systems with \emph{mean-field type dynamics}, under suitable conditions (see Def. \ref{def:liescheffer_structure}), admit a dimensional reduction. That is, the dynamics of the microscopic states of the system can be expressed in terms of \emph{macroscopic} states governed by a reduced system of differential equations. An overview of the mathematical objects involved, including microscopic and macroscopic systems and mean-field network dynamics, is provided in Fig.~\ref{fig:networkdynamics}.

We first define the notion of the \emph{mean-field Lie-Scheffers structure}.
\begin{definition} (Mean-field Lie-Scheffers Structure)
\label{def:liescheffer_structure}
    A system of differential equations \eqref{eq:microscopic_system_de} with mean-field type dynamics \eqref{eq:special structure} is said to have a \emph{mean-field Lie-Scheffers structure} if there exists a finite-dimensional real Lie algebra of vector fields $\vgl = \text{span}\{\Lg_1,\ldots,\Lg_r\}$ on $\mathbb{R}^\ds$ such that for any smooth function $\bm{\zeta}:\mathbb{R}\to\mathbb{R}^p$, the tangent vector field $f(\mc,\bm{\zeta}(t),t)
\cdot\partial_\mc$ (associated with the components of $\mathbf{F}$) can be written as
    \begin{equation}
        f(\mc,\bm{\zeta}(t),t)\cdot\partial_\mc = \sum_{\alpha=1}^r a_{\alpha}(t)\Lg_\alpha(\mc),
    \end{equation}
    where $a_1(t),\dots,a_r(t)$ are smooth scalar functions and $\Lg_\alpha$ are the basis generators of $\vgl$, and where we define $f\cdot \partial_\mc=f_1\partial_{x_1}+\cdots f_d\partial_{x_d}$.
\end{definition}
Note that even though the microscopic states, $\xx$, of a system with mean-field dynamics \eqref{eq:special structure}  are coupled to each other through the mean-field, $\mfmf$, the microscopic state equation effectively decouple, and the mean-field appears as a given background field $\bm{\zeta}^*(t)$ that is explicit in time. The following lemma formalizes this idea.

\begin{lemma}[Uncoupling]
\label{lem:uncoupling theorem}
Let \eqref{eq:microscopic_system} be a system of differential equations with mean-field-type dynamics. If $\m^*(t)=(\mc^*_1,\ldots,\mc^*_\dn)$ is a solution of \eqref{eq:microscopic_system} then each of its components (i.e., $(\mc^*_1,\cdots,\mc^*_n))$ satisfy
\begin{equation}
    \mc' = f(\mc,\bm{\zeta}^*(t),t),
    \label{eq:x'=f(x,zeta,t)}
\end{equation}
where $\bm\zeta^*(t) =\mfmf(\m^*(t))$ is the background field considered to be only a function of $t$.
\end{lemma}
\begin{proof}
    Let $\m^*(t)$ be a solution of \eqref{eq:microscopic_system},
   \begin{equation}
       \frac{\d\mc_i^{*}}{\d t}(t) =
        f(\mc_i^*,\mfmf(\m^*),t), \quad i \in V.
   \end{equation}
Therefore, each component $\mc^*_i(t)$ of $\m^*(t)$ satisfies
\begin{align}\label{eq:single_node_frozen}
    \mc_i' =f(\mc_i,\bm{\zeta}^*(t),t),
\end{align}
where we have set $\mfmf(\m^*(t))=:\bm\zeta^*(t)$ as the background field.

\end{proof}

A few remarks may be in order:

{\it Remark.}
\begin{enumerate}
\item If a system \eqref{eq:microscopic_system} has mean-field-type dynamics, then every oscillator obeys the same differential equation \eqref{eq:special structure}.  The vector field thus appears to decompose as $\mathbf{F}=f\oplus f\oplus \ldots\oplus f$ with decoupled nodal dynamics, but the nodes are in fact coupled through $\mfmf$.
This lemma shows that along any trajectory, the mean-field coupling acts as a time-dependent forcing term $\bm\zeta^*(t)$. While the microscopic states remain coupled through $\mfmf$, each component $\mc_i^*(t)$ can be treated as a solution of the single-node equation~\eqref{eq:single_node_frozen}, thus allowing Lie-Scheffers theory to be applied node-to-node and subsequently lifted to the full network.

\item To be clear,  the lemma does not state that the microscopic states $\mc_i^*(t)$ can be \emph{constructed} from~\eqref{eq:single_node_frozen}, because computing $\bm\zeta^*(t)$ would require prior knowledge of the full network state $\m^*(t)$. Rather, we assert that every component of $\m^*$ \emph{satisfies} equation \eqref{eq:single_node_frozen}.

\item Since $\bm\zeta^*(t) = \mfmf(\m^*(t))$ depends on the solution $\m^*(t)$ itself, its explicit form is generally unknown. However, standard ODE existence theorems guarantee such a function exists along any trajectory, even if we may not be able to explicitly calculate it.
\end{enumerate}

\medskip

We now state and prove the main result of our work.
\begin{theorem}[Main result: Lie-Scheffers theorem for systems of mean-field type]\label{thm:main_result} Let \eqref{eq:microscopic_system} be a microscopic system of differential equations with mean-field Lie-Scheffers structure.  Then the solutions $\m^*(t)$ of the microscopic system can be expressed
\begin{equation}
    \m^*(t)=\Csf(\M(t),\bm\Xi),
    \label{eq:z*=Phi}
\end{equation}
where $\Csf:\mathbb{R}^{m\ds}\times  \mathbb{R}^{\dn\ds}\to\mathbb{R}^{\dn\ds}$
is the \emph{collection of superposition rules}
defined by
$$\Csf:=(\Sf(\M,\xixi_1),\ldots,\Sf(\M,\xixi_n))^T.$$
The collection of superposition rules is parametrized by  the \emph{macroscopic variables} $\M:=(\mathbf{\Mc}_1,\ldots, \mathbf{\Mc}_m)^T\in \mathbb{R}^{m\ds}$ (corresponding to the \emph{fundamental solutions} $\mathbf{x}^{(1)}, \dots, \mathbf{x}^{(m)}$ introduced in Sec.~\ref{sec:primer}), which satisfy the \emph{macroscopic system} of differential equations,
\begin{equation}
    \M' = \mathbf{G}(\M,t),\quad\M(0)=\M_o\in\mathbb{R}^{m\ds},
    \label{eq:q'=G}
\end{equation}
where the function $\mathbf{G}:\mathbb{R}^{m\ds}\times \mathbb{R}\to\mathbb{R}^{m\ds}$ is defined by \eqref{eq:G} .
where $\M_o$ is a constant vector setting initial conditions on $\M(t)$.
The \emph{leaf parameter} $\bm\Xi := (\xixi_1,\ldots,\xixi_\dn) \in \R^{\dn\ds}$ with $ \xixi_i\in\R^\ds$ is a constant vector determined by initial conditions on $\m^*(t)$.
\end{theorem}
\begin{proof}
    Let $\m^*(t)$ be the solution of \eqref{eq:microscopic_system}. By Lemma~\eqref{lem:uncoupling theorem}, the components of $\m^*(t)$, solve
    \begin{equation}
    \mc' = f(\mc,\bm{\zeta}^*(t),t).
    \label{eq:z'=f}
    \end{equation}
    The mean-field Lie-Scheffers structure gives a superposition rule for the solutions of \eqref{eq:z'=f} via Theorem~\ref{thm:Lie theorem}. Therefore, the components of the solution vector of \eqref{eq:microscopic_system} can be written as a superposition of solutions $\M:=(\mathbf{\Mc}_1,\ldots, \mathbf{\Mc}_m)^T$ of \eqref{eq:z'=f} (for some $m$). That is, there exists a collection of superposition functions $\Sf:\mathbb{R}^m\times\mathbb{R}^\ds\to\mathbb{R}^\ds$, such that
    \begin{equation}
            \mc^*_i(t) = \Sf(\M(t),\xixi_i),
            \label{eq:z*i=Psi}
    \end{equation}
    where $i\in V$ and $\mathbf{\Mc}_1,\mathbf{\Mc}_2,\ldots,\mathbf{\Mc}_m$, are \emph{fundamental solutions} that solve \eqref{eq:z'=f} with leaf parameter $\bm\Xi:=(\bm{\xi}_1,\ldots, \bm{\xi}_n)^T$ which is determined by the initial conditions on ${\XX_i}^*(t)$.

    Defining $\Csf(\M,\bm\Xi):=(\Sf(\M,\xixi_1),\ldots,\Sf(\M,\xixi_n))^T$, we establish \eqref{eq:z*=Phi}. We now prove the existence of $\mathbf{G}$, cf.~\eqref{eq:q'=G}.

    Since $\mathbf{\Mc}_1,\ldots,\mathbf{\Mc}_m$ are \emph{fundamental} solutions of \eqref{eq:z'=f}, they must satisfy the said differential equation. It then follows that
    \begin{equation}\label{eq:macroscopic_system1}
        \M' =
        \begin{pmatrix}
        \mathbf{\Mc}_1' \\     \mathbf{\Mc}_2'\\ \vdots \\ \mathbf{\Mc}_m' \
        \end{pmatrix}
        =
        \begin{bmatrix}
            f(\mathbf{\Mc}_1, \bm\zeta^*(t),t)\\
            f(\mathbf{\Mc}_2, \bm\zeta^*(t),t)\\
            \vdots\\
            f(\mathbf{\Mc}_m, \bm\zeta^*(t),t)
        \end{bmatrix}.
    \end{equation}
    Recall that, as defined in Lemma~\eqref{lem:uncoupling theorem}, $\bm{\zeta }^*(t)=\mfmf(\m^*(t))$. Using \eqref{eq:z*=Phi} to write $\m^*$ in terms of $\M$, i.e., $\bm\zeta^*(t) =\mfmf(\bm\Phi(\M(t),\bm \Xi))$, we see that $\M$ satisfies the equation
    \begin{equation}
    \label{eq:G}
        \M' = \mathbf{G}(\M,t)=\begin{bmatrix}
            f(\mathbf{\Mc}_1, \mfmf(\bm\Phi(\M(t),\bm\Xi)),t)\\
            f(\mathbf{\Mc}_2, \mfmf(\bm\Phi(\M(t),\bm\Xi)),t)\\
            \vdots\\
            f(\mathbf{\Mc}_m, \mfmf(\bm\Phi(\M(t),\bm\Xi)),t)
        \end{bmatrix},
    \end{equation}
    where we call the right-hand side of the previous equation $\mathbf{G}$. Note that this step yields \emph{closure} of the equations, since the microscopic dynamics in the mean-field $\mf(\cdot)$ can be expressed in terms of $\M$ via the superposition rule $\Sf$.

    Also note that $\mathbf{G}$ depends on $\bm \Xi$ explicitly through $\mfmf(\bm\Phi(\M(t),\bm \Xi))$. This means the macroscopic system is parameterized by the leaf parameters, distinguishing different invariant manifolds.

\end{proof}

\subsection{Dimensional reduction.}
We summarize our results and put them into context, see also Fig.~\ref{fig:networkdynamics}. We have shown that the microscopic network system  \eqref{eq:microscopic_system} in $\m\in\R^{\dn\ds}$ is exactly described by the macroscopic system  \eqref{eq:macroscopic_system1} in $\M\in\R^{m\ds}$, provided that the microscopic system~\eqref{eq:microscopic_system} possesses mean-field Lie-Scheffers structure (Def.~\ref{def:liescheffer_structure}). Mean-field Lie-Scheffers structure implies that the microscopic dynamics are of mean-field type  (Def.~\ref{def:meanfieldtype_dynamics}), that is, the dynamics associated to every network node are identical. In particular, the mean-field Lie-Scheffers structure implies that the vector field associated with every individual node is generated by a Vessiot-Guldberg algebra. If this condition is met, there exists  a collection of superposition rules, $\m(t)=\Csf_{\M(t)}(\bm\Xi)\in\R^{\dn\ds}$, where the parameters $\M(t)\in\R^{m\ds}$ satisfying \eqref{eq:macroscopic_system1} generate the flow of the microscopic system, given a leaf parameter $\bm\Xi$. Provided  $m<\dn$ (i.e., the number of nodes in the network exceeds  the number of fundamental solutions required by the nodes), the macroscopic description via $\M(t)$ is \emph{lower dimensional}.

Unlike for the nodal system, note that we do not require the microscopic system \eqref{eq:microscopic_system} (as a whole) to possess Lie-Scheffers structure.
If we did, we would expect the number of required fundamental solutions to surpass the dimensionality of the associated macroscopic system~\eqref{eq:macroscopic_system1} substantially; clearly, this would undermine our goal of a dimensional reduction. Instead, our requirements are weaker, and $m$ fundamental solutions solve the $m\ds$-dimensional macroscopic system.

\subsection{Structure of the microscopic phase space: Foliation, initial conditions, and constants of motion}

The flows of the micro- and macroscopic systems are linked through the collection of superposition rules \eqref{eq:z*=Phi},  $\m(t)=\Csf_{\M(t)}(\bm\Xi)$. Consequently, the microscopic flow is confined to an $m\ds$-dimensional manifold, denoted by $M_{\bm\Xi}\subset \R^{\dn\ds}$, as shown in Fig.~\ref{fig:networkdynamics}(b). This defines a continuous family of invariant manifolds, parameterized by the leaf parameter $\bm\Xi\in\R^{\dn\ds}$, and forms a \emph{foliation of the microscopic space} into $m\ds$-dimensional leaves $M_{\bm\Xi}$ shown in Fig.~\ref{fig:networkdynamics}(b).

The entries of the leaf parameter $\bm\Xi $ may be seen as analogous to integration constants and can --- as such --- \emph{a priori} be chosen freely. However, if initial data of both micro- and macroscopic system are given, the leaf parameter is constrained by its compatibility with the collection of superpositions, i.e.,
\begin{equation}\label{eq:superpostion_initial}
    \m(0)=\Csf(\M(0),\bm\Xi).
\end{equation}
Thus, admissible initial data determine the corresponding invariant manifold $M_{\bm\Xi}$. In other words, the continuous family of invariant manifolds $M_{\bm\Xi}$ is then parameterized by the initial data via the superposition rule.

Two complementary perspectives arise:
\begin{itemize}
    \item [i)]{\it Fixed macroscopic initial data.}
    For prescribed $\M(0)$, the leaf parameter $\bm\Xi$ determines a compatible microscopic initial condition via \eqref{eq:superpostion_initial}, and hence selects the invariant manifold $\M(\bm\Xi)$ on which the microscopic dynamics evolves.

    \item [ii)] {\it Fixed microscopic initial data.}
    Alternatively, one may prescribe $\m(0)$ and regard \eqref{eq:superpostion_initial} as an equation for $\bm\Xi$. Because the superposition rule $\Sf(\M(0),\cdot)$ is locally invertible\footnote{
The superposition rule is locally invertible iff the fundamental vector fields of the associated Vessiot-Guldberg Lie algebra are linearly independent at a generic point, making the map a local diffeomorphism in the constant variables.}, the leaf parameter is determined by
    \begin{align}
       \bm\Xi &=\hat \Csf_{\M(0)}(\m(0)),
    \end{align}
    where
    \begin{align*}
    \hat \Csf_{\M(0)}:=(\Sf^{-1}(\m(0),\M(0)),\ldots,\Sf^{-1}(\m(0),\M(0))).
    \end{align*}
    Thus, $\bm\Xi$ becomes a function of the microscopic initial condition (and the chosen macroscopic initial data).
\end{itemize}

Since the leaf parameter $\bm\Xi$ is preserved along trajectories contained in the invariant manifold $M_{\bm\Xi}$, its entries are constant along the microscopic flow on that manifold. Hence, they represent \emph{constants of motion} and are analogous to integration constants.

The leaf parameter $\bm\Xi$ has $\dn\ds$ components; however, given initial condition $\M(0)$, the compatibility condition \eqref{eq:superpostion_initial} imposes $m\ds$ constraints, so that not all components of the leaf parameter are independent. Thus, recalling that  $\Csf:\mathbb{R}^{m\ds}\times  \mathbb{R}^{n\ds}\to\mathbb{R}^{n\ds}$, the number of independent constants of motion, $\gamma$, equals the dimension of the leaf parameter $\bm\Xi$ (i.e., the dimension of the microscopic system) minus the dimension of the macroscopic initial data:
\begin{align}\label{eq:nr_COM}
    \gamma &=
    % \dn\ds - m\ds=
    \ds(\dn-m).
\end{align}
To understand the physical meaning of $\gamma$, consider the Kuramoto model ($\ds=1$) with $\dn$ oscillator nodes. This system corresponds to an ensemble of Riccati equations governed by Möbius group symmetry, generated by $m=3$ fundamental solutions. In a generic uncoupled system, one would expect $\dn$ independent integration constants per trajectory. However, because the mean-field Lie-Scheffers structure constrains every node to the same set of $m=3$ fundamental solutions, the network flow is confined to invariant manifolds of dimension $3$.
Consequently, only $\gamma = \dn - 3$ \emph{independent} constants remain to distinguish between different invariant manifolds (leaves). When $\M(0)$ is fixed, choosing a specific leaf corresponds to fixing these $\dn-3$ free parameters. It is crucial to distinguish this from the trivial fact that \emph{any} ODE preserves all its initial conditions along a specific orbit; here, $\gamma$ counts the constants that are independent \emph{across the foliation}, serving to label the distinct manifolds rather than just tracking a single trajectory.

This quantity $\gamma$ is typically referred to in the literature as the number of ``constants of motion'' for the Kuramoto model; however, the more precise terminology is \emph{independent} constants of motion.

Selecting $\bm\Xi$ sets the particular leaf (invariant manifold) on which the dynamics evolves. If one moves the dynamics to a neighboring leaf while keeping the initial condition $\M(0)$ fixed, one effectively changes these $\dn-m$ free parameters, selecting a different family of trajectories consistent with the same macroscopic initial state.
\subsection{Compatibility of mean-field functions}\label{sec:compatible_mean_field_functions}
The mean-field function $\mfmf:\mathbb{R}^{\ds\dn}\to\mathbb{R}^p$ (for some $p\in\N$) must be compatible with mean-field type dynamics, see Def.~\eqref{def:meanfieldtype_dynamics}. In particular, identical nodal dynamics implies that the mean-field function $\mfmf$ must be the same for every node $ i\in V$, i.e., the mean-field must not vary with node $i\in V$; e.g., summing over different subsets of nodes: $\mfmf_i=\sum_{j\in V_i\subset V} \mc_j(t)$ is not admissible.
We outline a few instances of  mean-field functions that are compatible with the framework:
\renewcommand{\labelenumi}{\roman{enumi})}
\begin{enumerate}
    \item The mean-field function is not required to depend on \textit{all} nodes $i\in V$. For example,  the linear scalar mean-field $\mu:\R^n\to\R$  with \(\mf(t)=\sum_{j\in V'\subset V} \x_j(t) = \x_1(t) + \x_3(t)\) where $\x_i\in\R$ is admissible.

    \item If the mean-field function is vector-valued, different subsets of $\{\x_1,\ldots,\x_n\}$ with $\x_i\in\R$ may contribute to different vector components of $\mfmf =(\mf_1,\ldots,\mf_p)$. For example, consider the mean-field function where $\mf_k=\sum_{j\in V'_k}\mc_j$ and $\mfmf=(\mf_1,0,\mf_3)=(\x_1+\x_3,0,\x_2+\x_3)$ .

    \item The mean-field function can be a weighted function with non-uniform coupling between the nodes, and also be nonlinear in terms of $x_i\in\R$, e.g., $\mf(t)=\sum_{j\in V} c_j (\x_j(t))^\beta$,  where $\beta\in \N$ and constant weights $c_i\in\R$ . All-to-all coupling (complete graph $K_\dn$) with uniform weights~\cite{WatanabeStrogatz1993,OttAntonsen2008} recovered as a special case with $\mf=\dn^{-1}\sum_{j=1}^\dn\x_j$.
\end{enumerate}

\section{Applications of the main result\label{sec:applications}}

We apply our framework of mean-field Lie-Scheffers theory, developed in Sec.~\ref{sec:main_result}, to three distinct classes of network dynamical systems, all of which exhibit exact dimensional reduction: ensembles of quasi-linear equations, Riccati-type ensembles, and ensembles of generalized Bernoulli equations; the last of which is a new class not considered in the literature. We demonstrate how the underlying mean-field Lie-Scheffers structure enables an exact dimensional reduction.

\subsection{Ensembles of quasi-linear equations}
\label{subsec:quasilinear_examples}

We begin by establishing the mean-field Lie-Scheffers structure for an ensemble of inhomogeneous quasi-linear ordinary differential equations. We then specialize to second-order systems to illustrate the dimensional reduction with a physically motivated example.
\smallskip
\paragraph{Microscopic equations and mean-field Lie-Scheffers structure.}
Consider an ensemble of $\dn$ nodes where each node $i$ obeys a $\ds$-th order inhomogeneous quasi-linear differential equation:
\begin{equation}\label{eq:quasilinear_general}
    \sum_{c=0}^{\ds} a_c( \mf(\mathbf{x}), t) \, u_i^{(c)} = b_i(\mf(\mathbf{x}), t),
\end{equation}
where $u_i^{(c)} = d^cu_i/dt^c$, $\mathbf{x} = (u_1, \ldots, u_\dn)^T$ is the network state, and $\mf(\mathbf{x})$ is a mean-field functional. The key structural assumption is that the coefficient functions $a_c$ and $b_i$ have \emph{identical functional form} across all nodes $i$ (Note that if inhomogeneous terms $b_i$ are non-identical across nodes $i$, one still obtains a moderate dimensional reduction, but the number of resulting macroscopic equations scales with $\dn$).

To apply our theory of Sec.~\ref{sec:main_result}, which applies to first-order systems, we reduce Eq.~\eqref{eq:quasilinear_general} to first order. For each node, define the local phase vector $\mathbf{y}_i = (u_i, u_i', \ldots, u_i^{(\ds-1)})^T \in \mathbb{R}^\ds$. Solving for the highest derivative (assuming $a_\ds \neq 0$), the system takes the vector form
\begin{equation}\label{eq:quasilinear_matrix_form}
    \mathbf{y}_i' = A(\mf(\mathbf{x}), t)\,\mathbf{y}_i + \mathbf{B}(\mf(\mathbf{x}), t),\;i=1,\ldots,n,
\end{equation}
where $A$ is an appropriate $d\times d$ coefficient matrix and $\mathbf{B} = (b_1,\ldots,b_d)^T$.
Let us define the following vector fields  $\Lg_{ij} = y_i \partial_{y_j}$ and $T_i = \partial_{y_i}$. These $\ds^2+\ds$  generators span the \emph{affine Lie algebra} $\vgl\cong\aff{\ds,\R} = \gl{\ds,\R} \ltimes \mathbb{R}^\ds$ (where $\ltimes$ denotes the semi-direct product). Closure under the Lie bracket follows from the computations $[\Lg_{ij},\Lg_{kl}]=\delta_{jk} \Lg_{il}-\delta_{li}\Lg_{kj}$, $[T_i,T_j]=0$, and $[\Lg_{ij},T_k]=-\delta_{ik}T_j$. Since the coefficients $A$, $B$, and the mean-field coupling $\mu$ are identical across all the nodes, it follows that \eqref{eq:quasilinear_matrix_form} admits a mean-field Lie-Scheffers structure. By Theorem~\ref{thm:main_result}, the system Eq.~\eqref{eq:quasilinear_matrix_form} admits a dimensional reduction where the microscopic states of the system can be computed from the macroscopic states. This replicates results from Augustsson \emph{et al.}~\cite{augustsson2025quasilinear}. Instead of stating the framework in generality, we specialize to a concrete example in the following for simplicity.
\smallskip
\paragraph{Specialization to second-order systems.}
We identify the macroscopic equations and consider $\ds=2$, so that each node satisfies:
\begin{equation}\label{eq:second_order_network}
    \ddot{u}_i + \alpha(\mf(\mathbf{u}(t)),t) \dot{u}_i + \beta(\mf(\mathbf{u}(t)),t) u_i = f(\mf(\mathbf{u}(t)),t),
\end{equation}
where $i=1,\cdots,n$, and $\alpha, \beta, f$ are given functions (identical for each node), that depend on the mean-field trajectory. The local phase space is $(u_i, v_i)^T \in \mathbb{R}^2$. The vector field spans  $\aff{2}$, requiring $m=3$ fundamental solutions for the general inhomogeneous case ($f\neq 0$). This can be easily seen by noting that along a solution trajectory, the mean-field becomes a known function of time, $\zeta^*(t):= \mf(\u(t))$,  which we substitute into  \eqref{eq:second_order_network}, as we do in the Uncoupling Lemma.~\ref{lem:uncoupling theorem}. The resulting differential equation has coefficients $\alpha(t):=\alpha(\zeta^*(t),t)$, cf.,~\eqref{eq:x'=f(x,zeta,t)}, and is then a \emph{linear} second-order differential equation which requires two homogeneous and one particular solutions to describe the general solution, or equivalently three general solutions. See Example 4 in Ref. \cite{carinena2007superposition}.

Thus, the dimension of the macroscopic system, \eqref{eq:macroscopic_system1} in the second-order case is $m\ds = 6$, and in particular independent of $\dn$. If instead $f$  varies across nodes, a distinct particular solution is needed per node, and the reduction advantage scales with the number of unique forcing profiles rather than vanishing entirely.
\smallskip
\paragraph{Network example: Damped Hill-Mathieu equations.}
We further specialize to a physically motivated case and consider an ensemble of damped Hill-Mathieu equations combining parametric driving, damping, and mean-field feedback:
\begin{equation}\label{eq:Hill_Mathieu_microscopic}
    \ddot{u}_i + \gamma \dot{u}_i + \left[ \omega_0^2 + p(t) + \alpha \, \mf(\mathbf{u}) \right] u_i = 0, \quad i=1,\ldots,\dn.
\end{equation}
Here, $\gamma > 0$ is constant linear damping, $p(t)$ is an external periodic drive, and $\alpha \mf(\mathbf{u})$ defines the collective stiffness with $\mf(\mathbf{u}) = \dn^{-1}\sum_{j=1}^\dn w_j u_j(t)$ with constant weights $w_j\in\R$. Note that the restoring force depends on both the global network state and time.

Since there is no independent forcing term in \eqref{eq:Hill_Mathieu_microscopic}, i.e., $f\equiv 0$ and therefore $\mathbf{B} \equiv 0$, the algebra reduces to $\mathfrak{gl}(2)$, requiring only $m=2$ fundamental solutions. This  contrasts with the inhomogeneous case considered in the previous part which required $m=3$. Let $\mathbf{\Mc}^{(1)}(t)=(U_1,V_1)^T$ and $\mathbf{\Mc}^{(2)}(t)=(U_2,V_2)^T$ be two linearly independent solutions of the representative macroscopic equations, compare with~\eqref{eq:macroscopic_system1}:
\begin{equation}\label{eq:Hill_Mathieu_macroscopic}
    \begin{pmatrix} U_\ell \\ V_\ell \end{pmatrix}' = \begin{pmatrix} 0 & 1 \\ -[\omega_0^2 + p(t) + \alpha {\zeta}^*(t)] & -\gamma \end{pmatrix} \begin{pmatrix} U_\ell \\ V_\ell \end{pmatrix},\quad \ell=1,2.
\end{equation}
Comparing to Theorem~\ref{thm:main_result}, we identify $\zeta^*(t) =\mu(c_{i,1} \mathbf{\Mc}^{(1)}(t) + c_{i,2} \mathbf{\Mc}^{(2)}(t))$ with constants $c_{i,\ell}$ fixed by initial conditions; the linear combination appearing in $\mu$ is the superposition rule $\Sf$ for the microscopic Hill-Mathieu equation. Substituting $\mathbf{\Mc}^{(i)}$ in terms of $(U_i,V_i)^T$, for $i=1,2$ into the mean-field definition yields the following explicit expression in terms of macroscopic variables:
\begin{equation}
    \zeta^*(t) = C_1 U_1(t) + C_2 U_2(t), \quad C_\ell := \frac{1}{\dn}\sum_{i=1}^\dn w_ic_{i,\ell},\quad \ell=1,2.
\end{equation}
Thus, we have obtained a closed $4$-dimensional macroscopic system for $(U_1, V_1, U_2, V_2)$, corresponding to the macroscopic equations \eqref{eq:q'=G} of Theorem~\ref{thm:main_result}.
As a result we have reduced   the microscopic  $2\dn$-dimensional equations~\eqref{eq:Hill_Mathieu_microscopic} to exactly $m\ds = 4$-dimensional macroscopic equations~\eqref{eq:Hill_Mathieu_macroscopic}, independent of network size $\dn$. Thus, the microscopic Hill equations possess $\gamma = 2\dn - 4$ independent constants of motion, defining the invariant leaves of the foliated phase space.

The $2\dn$ integration constants ${c_{i,\ell}}$ are defined only up to a $GL(2)$ change of basis in the choice of fundamental solutions. Fixing this freedom --- e.g., by prescribing four initial conditions for $(U_\ell, V_\ell)$ --- eliminates the four `gauge' parameters $c_{i,\ell}$ ($i=1,2; \ell=1,1)$, leaving $\gamma = 2\dn - 4$ independent quantities that label the leaves.
To put it concretely, one may take the leaf parameters to be the residual coefficients $(c_{k,1}, c_{k,2})$ for $k = 3, \ldots, \dn$ after setting the coefficients for two reference nodes to $(c_{1,1},c_{1,2}) = (1,0)^T$ and $(c_{2,1},c_{2,2}) = (0,1)^T$, with the mean-field couplings $C_\ell = \frac{1}{\dn}\sum_i w_i c_{i,\ell}$ then determined by these invariants alone.

Furthermore, a $GL(2)$ change of basis on the fundamental solutions transforms both the macroscopic variables $(U_\ell, V_\ell)$ and the coefficients $C_\ell$, so the \emph{form of the macroscopic equations changes} while their solutions --- and hence the microscopic dynamics --- remain invariant. This highlights that the reduction yields a family of equivalent macroscopic representations, all generating the same network flow.

Finally, note that these results for quasilinear ensembles extend to the complex number field with the algebra $\gl{\ds,\C}$. For details and other example applications see Ref.~\cite{augustsson2025quasilinear}.

\subsection{Ensembles of Riccati equations: The Kuramoto model}
\label{subsec:riccati_examples}

We next apply our theory to networks of Riccati-type equations, demonstrating how the Möbius symmetry of single nodes translates into exact dimensional reductions for coupled oscillator populations~\cite{WatanabeStrogatz1993,Marvel2009}.
As established in Sec.~\ref{sec:examples_lie_scheffers_theory}, a scalar real-valued Riccati equation admits a superposition rule \eqref{eq:gen sol riccati}, governed by the $\mathfrak{sl}(2,\mathbb{R})$ algebra. This extends to complex-valued Riccati with $\sl{2,\C}$,  see Refs.~\cite{CestnikMartens2024,blazquezsanz2012}.
\smallskip
\paragraph{Microscopic equations and mean-field Lie-Scheffers structure.} Consider $\dn$ nodes with states $z_i(t)\in \C$ obeying first-order differential equations of Riccati type:
\begin{equation}\label{eq:riccati_network_microscopic}
    \dot{z}_i(t) = A(\mf(\mathbf{z}), t) + B(\mf(\mathbf{z}), t) z_i + C(\mf(\mathbf{z}), t) z_i^2, \quad i\in V,
\end{equation}
where $\mf(\mathbf{z})$ is a mean-field functional such as $\mf(\mathbf{z}) = \dn^{-1}\sum_{j=1}^\dn z_j$. Importantly, the coefficient functions $A, B, C$ must be identical across all nodes. Any dependence on global state enters parametrically through the mean field.

Note that \eqref{eq:riccati_network_microscopic} is, strictly speaking, not a Riccati equation because of the dependence of the coefficients on $\mathbf{z}$. Thus, the Lie-Scheffers Theorem~\eqref{thm:Lie theorem} cannot be applied to this equation as the vector-field generating the flow do not form a VGL algebra. However, in light of the Uncoupling Lemma~\ref{lem:uncoupling theorem}, along any solution trajectory $\mathbf{z}(t)$, the mean field becomes a known function of time, $\zeta^*(t) := \mf(\mathbf{z}(t))$. Consequently, the coefficients effectively become explicit time-dependent functions:
\[
    a(t) := A(\zeta^*(t), t), \quad b(t) := B(\zeta^*(t), t), \quad c(t) := C(\zeta^*(t), t).
\]
Under this consideration, each node satisfies the same time-dependent Riccati equation:
\begin{equation}
    \dot{z}_i = a(t) + b(t) z_i + c(t) z_i^2,
\end{equation}
which does form a Lie system. Thus, the system of equations \eqref{eq:riccati_network_microscopic},
has a mean-field Lie-Scheffers structure.

\smallskip
\paragraph{Superposition rule and geometry.}
As a consequence of Theorem \eqref{thm:main_result}, the system admits a dimensional reduction whereby the microscopic states $z_i(t)$ can be reconstructed from macroscopic states. The general solution for any node $i$ is given by the cross-ratio superposition rule (cf.\eqref{eq:gen sol riccati}, \eqref{eq:z*=Phi}:
\begin{equation}\label{eq:riccati_superposition_general}
    z_i(t) = \frac{\Mc_2(\Mc_1 - \Mc_3) - \xi_i \, \Mc_3(\Mc_1 - \Mc_2)}{(\Mc_1 - \Mc_3) - \xi_i \, (\Mc_1 - \Mc_2)},
\end{equation}
parameterized by fundamental solutions  $q_j(t)$ which satisfy the macroscopic Riccati equations~\eqref{eq:riccati_macroscopic_equations} and where $\xi_i \in \mathbb{C}$ is a constant leaf parameter determined by initial conditions. This quantity represents the conserved cross-ratio $(z_i, \Mc_1; \Mc_2, \Mc_3)$ associated with the Möbius group action. Geometrically, the phase space $\mathbb{C}^\dn$ is foliated into 3-dimensional (complex) orbits of the Möbius group, parameterized by the leaf parameters $\bm\Xi =( {\xi_i})_{i=1}^\dn$.
\smallskip
\paragraph{Macroscopic reduction and constants of motion.}
Substituting Eq.~\eqref{eq:riccati_superposition_general} into the definition of the mean-field yields explicit closure equations \eqref{eq:macroscopic_system1} for the macroscopic variables. For the complex order parameter $\mf = \dn^{-1}\sum_{j=1}^\dn z_j$, we define
\begin{equation}
    \mf^* = \frac{1}{\dn} \sum_{j=1}^\dn \frac{\Mc_2(\Mc_1 - \Mc_3) - \xi_j \, \Mc_3(\Mc_1 - \Mc_2)}{(\Mc_1 - \Mc_3) - \xi_j \, (\Mc_1 - \Mc_2)}.
\end{equation}
This allows construction of a closed 3-dimensional (complex) system for the fundamental solutions via the macroscopic equations, compare with Eqs.~\eqref{eq:q'=G},:
\begin{align}\label{eq:riccati_macroscopic_equations}
    \dot{\Mc}_k &= a(\mf^*(\Mc; \bm\Xi), t) + b(\mf^*(\Mc; \bm\Xi), t)\Mc_k + c(\mf^*(\Mc; \bm\Xi), t)\Mc_k^2,
\end{align}
where $k=1,2,3$.

Crucially, this reduction is exact for any finite $\dn$, transforming the original $\dn$-dimensional problem (over $\R$ or $\C$ ) into a $3$-dimensional system, independent of network size $\dn$, with $\gamma= \dn-3$ (independent) constants of motion. Thus, as a direct consequence of the Lie--Scheffers framework, we recover the classical result  by Watanabe and Strogatz~\cite{WatanabeStrogatz1993,WatanabeStrogatz1994} for real Riccati ensembles and the result of Cestnik and Martens~\cite{CestnikMartens2024} for complex Riccati ensembles. (Over $\C$, the $\dn-3$ constants of motion are complex-valued, corresponding to $2(\dn-3)$ real degrees of freedom.)
\smallskip
\paragraph{Application: The Kuramoto model.}
Consider the Kuramoto model for $\dn$ identical phase oscillators with sinusoidal coupling:
\begin{equation}\label{eq:kuramoto_original}
    \dot{\theta}_i = \omega + \frac{K}{\dn} \sum_{j=1}^\dn \sin(\theta_j - \theta_i),
\end{equation}
where $\omega$ is the common natural frequency and $K$ the coupling strength. Using the circle embedding $z_i = e^{i\theta_i} \in \mathbb{S}^1 \subset \mathbb{C}$ and the identity $\sin(\theta_j - \theta_i) = \text{Im}(e^{i(\theta_j-\theta_i)})=\text{Im}(z_i \bar z_j)$, Eq.~\eqref{eq:kuramoto_original} transforms into:
\begin{equation}
\label{eq:ricatti form of kuramoto}
    \dot{z}_i = i\omega z_i + \frac{K}{2}(\mf^* - \bar{\mf}^* z_i^2), \quad \text{with } \mf^*(t): = \dn^{-1}\sum_{j=1}^\dn z_j(t).
\end{equation}
This is precisely the canonical Riccati form \eqref{eq:riccati_network_microscopic} with coefficients:
\[
    a(t) =  \frac{K}{2}\mf^*(t), \quad b(t) = i\omega , \quad c(t) = -\frac{K}{2}\bar{\mf}^*(t).
\]
All nodes share identical $a$, $b$, $c$ depending only on $\mf^*(t)$, satisfying the uniformity condition. The reduction yields $m=3$ fundamental solutions evolving in $\mathbb{C}^3$, with $\dn-3$ independent cross-ratio constants $\xi_i \in\C$.

For practical implementation, one computes initial leaf parameters via:
\[
    \xi_i = \frac{(z_i(0) - \Mc_2(0))(\Mc_1(0) - \Mc_3(0))}{(z_i(0) - \Mc_3(0))(\Mc_1(0) - \Mc_2(0))},
\]
which can be simplified by choosing a projective basis $\Mc_1(0)=0, \Mc_2(0)=1, \Mc_3(0)=-1$ yielding $\xi_i = (1-z_i(0))/(1+z_i(0))$. The original $\dn$-oscillator dynamics are then recovered exactly from the 3D macroscopic evolution \eqref{eq:riccati_macroscopic_equations} via the superposition rule in \eqref{eq:riccati_superposition_general}.

In the calculation above, we have used a circle embedding to lift the phase variable  $\theta_i$ in the Kuramoto model \eqref{eq:kuramoto_original} to the complex plane to obtain the microscopic dynamics in Ricatti form \eqref{eq:ricatti form of kuramoto} in terms of $z_i\in\C$; however, we note that \eqref{eq:kuramoto_original} has a mean-field Lie-Scheffers structure, and therefore it is also possible to apply the framework developed in Sec. \ref{sec:main_result} directly to the phase description of the Kuramoto model to derive the corresponding macroscopic equations.

\smallskip
\paragraph{Extensions.}
This framework unifies several models of the general form $\dot \theta_j = a +  b e^{i\theta_j} +\bar be^{-i\theta_j}$ that are mean-field Lie-Scheffers systems amenable to our framework. The Theta neuron model~\cite{ermentrout1986parabolic}, viz. $\dot{\theta} = (1-\cos\theta) + I(t)(1+\cos\theta)$, maps to the inhomogeneous Riccati equation $\dot{z} = I(t) + z^2$ via the stereographic projection $z = \tan(\theta/2)$~\cite{laing2018dynamics}. Another model is a special case of the Winfree model, for details see Refs.~\cite{ariaratnam2001phase,Bick2018c}.
\smallskip
\paragraph{Connection to Cestnik-Martens parameterization.\label{subsec:Cestnik-Martens-parameterization}}
The macroscopic variables $(q_1,q_2,q_3)$ satisfy the canonical macroscopic equations~\eqref{eq:riccati_superposition_general}, which all are of Riccati type; by contrast, in Ref.~\cite{CestnikMartens2024} the macroscopic variables are $(Q,y,s)$, where only $Q$ satisfies a Riccati-type equation, but $y$ and $s$ satisfy ODEs of other type, and the superposition rule takes on the form $z_i=Q+(y\xi_i)/(1+s\xi_i)$.

The canonical macroscopic variables $(\Mc_1,\Mc_2,\Mc_3)$ in Eq.~\eqref{eq:riccati_superposition_general} can be related to the parameterization  $(Q,y,s)$ via the rational maps: $\Mc_1=Q$, $\Mc_2=Q+y/s$, $\Mc_3=Q+y/(1+s)$, with inverse $Q=\Mc_1$, $s=(\Mc_3-\Mc_1)/(\Mc_2-\Mc_3)$, $y=(\Mc_2-\Mc_1)(\Mc_3-\Mc_1)/(\Mc_2-\Mc_3)$. Likewise, the leaf parameters , cf.\eqref{eq:riccati_superposition_general}, are related by a Möbius transformation, $\xi_{i} = 1/(1-\xi_{\text{CM}})$.

The different forms of the macroscopic equations and superposition rules represent a different coordinate choice on the same Möbius orbit. Both macroscopic formulations generate identical flows for the microscopic dynamics on the same leaf. This observation reflects a broader theme of coordinate freedom on group orbits, as we will discuss in Sec.~\ref{sec:discussion}.

\subsection{Ensembles of generalized Bernoulli equations}
\label{subsec:bernoulli_examples}

We conclude our examples with ensembles of Bernoulli-type equations, which exhibit a simpler algebraic structure ($\aff{1}$) than the Riccati-type equations but accommodate a broader class of nonlinearities through the power-law exponent $p$.
\smallskip
\paragraph{Microscopic equations and Lie-Scheffers structure.}
Consider $\dn$ nodes with states $x_i(t) \in \mathbb{R}^+$ governed by generalized Bernoulli equations with mean-field feedback:
\begin{equation}\label{eq:bernoulli_network_microscopic}
    \dot{x}_i(t) = A(\mf(\mathbf{x}), t)\, x_i + B(\mf(\mathbf{x}), t)\, x_i^p, \quad i=1,\ldots,\dn,
\end{equation}
where $p \in \mathbb{R}$, $p \neq 1$, and $\mf(\mathbf{x})$ is a mean-field functional. As for the ensembles of quasi-linear and Riccati equations, the coefficient functions $A$ and $B$ must be identical across all nodes for dimensional reduction. We also assume that initial conditions are chosen such that the system admits a solution.
Note that the case where $p=3$ arises in several domains where nodes interact through a shared resource
or global feedback. Potential applications may include autocatalytic reactors with $\dn$ coupled well-mixed reactors~\cite{gray1983autocatalytic}.

The tangent vector fields $\Lg_1 = x\partial_x$ (dilation) and $\Lg_2 = x^p\partial_x$ (nonlinear growth) satisfy
\begin{equation}\label{eq:generalized_bernoulli_liebracket}
    [\Lg_1, \Lg_2] = (p-1)\Lg_2,
\end{equation}
thus spanning a $r=2$ dimensional VGL algebra $\vgl\cong \aff{1}$.
It follows that \eqref{eq:bernoulli_network_microscopic} admits a mean-field Lie-Scheffers structure.

It is straightforward to verify that $\vgl\cong \aff{1}$. The affine algebra $\aff{1}$ is spanned by $\Lg_D=x\partial_x$ (dilation) and $\Lg_T=\partial_x$ (translation) and generates the group $\text{Aff}(1) = \{x \mapsto \alpha x + \beta : \alpha, \beta\in\mathbb{R},  \alpha \neq 0\}$. Comparing the  associated Lie bracket structures, $[\Lg_D,\Lg_T]=-\Lg_T$, with \eqref{eq:generalized_bernoulli_liebracket} prompts the isomorphism $\phi(\Lg_D) = -\frac{1}{p-1}\Lg_1$ and $\phi(\Lg_T) = \Lg_2$, and so, $\vgl\cong \aff{1}$.

Note that we do not necessarily restrict ourselves to integer values of $p$. We exclude $p=1$ as this would lead to a trivial $r=1$ dimensional algebra. For $p=2$ we recover the homogeneous Riccati equation discussed earlier; see  Sec.~\ref{sec:examples_lie_scheffers_theory} and Fig.~\ref{fig:foliations}.
\smallskip

\paragraph{Superposition rule via linearization.}
Instead of constructing the superposition rule through the Frobenius-based PDE approach (Theorem~\ref{thm:foliation}), the Bernoulli equation admits an elegant alternative: the nonlinear transformation $y_k = x_k^{1-p}$ maps Eq.~\eqref{eq:bernoulli_network_microscopic} to the affine first-order linear equation:
\begin{equation}\label{eq:bernoulli_linearized}
    \dot{y}_i = (1-p)\,a(t)\, y_i + (1-p)\,b(t),
\end{equation}
where $a(t) := A(\zeta^*(t), t)$ and $b(t) := B(\zeta^*(t), t)$ along the mean-field trajectory, where $\zeta^*(t)$ is as defined in Lemma~\ref{lem:uncoupling theorem}.

Unlike the stereographic projection used in the case of Riccati equations, this map is a genuine diffeomorphism on $\R^+$, so the linearization is globally well-defined wherever the Bernoulli flow remains finite.

Eq.~\eqref{eq:bernoulli_linearized} is precisely an inhomogeneous first order linear system, requiring $m=2$ fundamental solutions $y^{(1)}(t)$ and $y^{(2)}(t)$.
By linearity, the general solution can be written as a \textit{linear superposition}, $y_i(t)=\xi_i\,y^{(1)}(t)+(1-\xi_i)\,y^{(2)}(t)$, where we identify the leaf parameter $\xi_i=y_i(0)=x_i(0)^{1-p}$. Defining Bernoulli fundamental solutions via the inverse map $x^{(\ell)}(t)=(y^{(\ell)}(t))^{1/(1-p)}, \ell = 1,2$ and substituting into the linear superposition yields the \textit{nonlinear superposition} principle:
\begin{equation}\label{eq:bernoulli_superposition}
    x_i(t) = \left[\xi_i \left(\Mc^{(1)}(t)\right)^{1-p} + (1-\xi_i)\left(\Mc^{(2)}(t)\right)^{1-p}\right]^{1/(1-p)},
\end{equation}
parameterized by the fundamental solutions, $x^{(k)}=q_k, \;k=1,2$, which satisfy the macroscopic Bernoulli equations:
\begin{align}\label{eq:Bernoulli_macroscopic}
    \dot{\Mc}_k = a(t)\, \Mc_k + b(t)\,\Mc_k^p,\quad k=1,2.\
\end{align}
\smallskip
\paragraph{Macroscopic reduction and constants of motion.}
Substituting the nonlinear superposition~\eqref{eq:bernoulli_superposition} into an admissible mean-field function, e.g., $\mf = \dn^{-1}\sum_{j=1}^\dn x_j $ , yields explicit closure in terms of the $m=2$ fundamental solutions. Thus, given that $\ds=1$, the resulting self-consistent macroscopic system~
\eqref{eq:Bernoulli_macroscopic} (compare with \eqref{eq:q'=G}) is $m\ds=2$ dimensional,  independent of network size $\dn$. The number of independent constants of motion is $\gamma = \dn - 2$, fewer than for the ensemble of Riccati equations ($\dn-3$), reflecting the simpler algebraic structure.

\section{Discussion\label{sec:discussion}}

\paragraph{Summary.} We have presented a unified framework for the exact dimensional reduction for network dynamical systems, based on Lie-Scheffers theory. Our main result is that a network with $\dn$ nodes equipped with identical nodal dynamics and mean-field coupling possesses a dimensional reduction from $\dn\ds$ to $m\ds$ dimensions whenever the nodal vector field is generated by a finite-dimensional Vessiot-Guldberg Lie algebra, requiring $m$ fundamental solutions. This reduction is exact for any finite $\dn$, and results in a foliation of phase space with invariant manifolds. These leaves of the foliation are labeled by $\gamma=\ds(\dn-m)$ independent constants of motion. Extending our framework to $S$ subpopulations~\cite{Abrams2008,montbrio2004synchronization, pikovsky2008partially} is straightforward and multiplies both the macroscopic dimension $m\ds$ and the constant count $\gamma$ by $S.$
\smallskip
\paragraph{Mechanism behind the reduction.} The answer to the question \textit{why} certain networks admit exact reductions is found in the Lie-Scheffers theorem (Theorem~\ref{thm:Lie theorem}, Ref.~\cite{carinena2011lie}), which gives necessary and sufficient conditions for the existence of a (linear or nonlinear) superposition rule for a system of differential equations. Lie-Scheffers theory guarantees that a superposition rule exists provided that the nodal vector field is involutive, i.e., its component vector fields close under Lie bracket to form a finite-dimensional Vessiot-Guldberg Lie algebra (Incidentally, the Lie-Scheffers theorem is a consequence of the Frobenius theorem in differential topology; see remarks following Theorem\ref{thm:Lie theorem}). Any solution can then be reconstructed from $m$ fundamental solutions via a time-independent superposition rule $\Sf$, confining the dynamics to a leaf of the foliated phase space. Consequently, if the additional conditions of a mean-field Lie Scheffers structure are met (see Theorem \ref{thm:main_result} in Sec.~\ref{sec:main_result}), for a network of $\dn$ nodes with identical nodal dynamics, a dimensional reduction from $\mathcal{O}(\dn)$ to $\mathcal{O}(1)$ is feasible.

\smallskip
\paragraph{Unified theory for known results and generalization.}
In 2009, Marvel, Mirollo and Strogatz~\cite{Marvel2009} identified the Möbius group as the `hidden' structure behind sinusoidally coupled oscillators, thus broadening the applicability of the reductions discovered by Watanabe and Strogatz~\cite{WatanabeStrogatz1993,WatanabeStrogatz1994} and Ott and Antonsen~\cite{OttAntonsen2008} . Our theory offers a broader context of their result: The Möbius group is the Lie group associated with the VGL algebra $\sl{2,\mathbb{R}}$---thus, the superposition rule is guaranteed by Lie-Scheffers theory rather than by the specific algebraic form of the coupled dynamical system.

Importantly, our theoretical framework significantly expands the class of network dynamical systems to exact dimensional reduction. We have shown how the framework recovers several known results as special cases. These include the reduction for Riccati ensembles (with VGL algebras $\sl{2,\R}$ and $\sl{2,\C}$) encompassing the Kuramoto model~\cite{WatanabeStrogatz1993,Marvel2009,CestnikMartens2024} and Theta neurons~\cite{ermentrout1986parabolic,laing2018dynamics,Bick2018c}; and the quasilinear ensembles (with VGL algebra $\aff{\ds}$)~\cite{augustsson2025quasilinear} . Furthermore, the Bernoulli ensemble  ($\aff{1}, m=2$) illustrates that the framework extends beyond Riccati-type dynamics, accommodating arbitrary nonlinearities (based on the value of $p$) and with fewer constants of motion ($\gamma = \dn -2$).

Thus, the framework provides (i) the common algebraic mechanism (superposition rule and VGL structure) that previously was established on a case-by-case basis; and (ii) a systematic method to explore other network dynamical systems for exact dimensional reduction. Notably, one may even envision \emph{reverse-engineering nodal dynamics} by designing vector fields whose VGL algebra exhibits a prescribed Lie structure with given symmetries, thereby yielding guaranteed reductions.
\smallskip

\paragraph{Exact description of microscopic dynamics.} The superposition rule $\Sf$ provides the exact functional relationship between the microscopic and macroscopic worlds: each state variable $\xx_i(t)$ is given as an explicit function of the $m$ macroscopic solutions $\M_k(t)$  and the $\gamma=\ds(\dn-m)$ labels $\xixi$.  This structure partitions the phase space into invariant manifolds (leaves). The dynamics on each leaf are governed by the same $m\ds$-dimensional macroscopic system, distinguished only by the constant leaf parameters.
\smallskip

\paragraph{Closure of the macroscopic system.} An important feature of the framework is that closure of the macroscopic dynamical system is automatic: all microscopic variables are eliminated by substituting the superposition rule into the mean-field functional, resulting in a self-consistent system in the macroscopic variables $\M_k(t)$ alone, c.f., \eqref{eq:G}.
\smallskip

\paragraph{Sufficient conditions.}
Two structural conditions (mean-field Lie-Scheffers structure, see Definition~\ref{def:liescheffer_structure}) are required for the reduction: (i) identical nodal dynamics, $f_i=f$ for all nodes $i$; and (ii) VGL algebra admitting a superposition rule. Together, they are sufficient.

Without the VGL algebra structure, even identical nodal dynamics leave the nodes fully \textit{entangled} through the mean-field $\mf$, and thus, possibly exploring the full $\R^{\dn\ds}$ phase space rather than confining them to a leaf. Conversely, heterogeneity in the nodal dynamics --- such as non-identical frequencies or non-uniform couplings --- destroys the shared VGL structure, and the different nodes obey different algebras without generating a common superposition rule, so that the exact finite-$\dn$ reduction is lost. In either case, the absence of a common superposition rule prevents a description of the nodal dynamics in terms of the same fundamental solutions, and trajectories do not remain confined to a shared invariant manifold.

Importantly, structures that couple to different graph neighborhoods break the homogeneity requirement (i)
; however, the mean-field, $\mf$, does not need to depend on all nodes but must be identical for every node. Thus, a certain class of non-uniformity is admissible; see Sec.~\ref{sec:compatible_mean_field_functions}.
\smallskip

\paragraph{Partial integrability under structured heterogeneity.}
A partial exception to the homogeneity requirement occurs in quasi-linear systems with heterogeneous forcing (see Sec.~\ref{subsec:quasilinear_examples} and Ref.~\cite{augustsson2025quasilinear}). Here, solutions of the homogeneous part of the solution remain shared across nodes, but each distinct nodal forcing requires its own inhomogeneous (particular) solution. In this case, there are $m$ shared homogeneous fundamental solutions but $\dn$ distinct particular solutions --- one per nodal forcing. However, the dimensionality still scales as $\mathcal{O}(\dn)$ so that the reduction is degraded from $\mathcal{O}(1)$ rather than eliminated. Lie-Scheffers theory gives exact conditions for the existence of the superposition; partial integrability is possible even when only a part of the algebra is shared across nodes.
\smallskip

\paragraph{Coordinate freedom and equivalent representations.}
The  superposition rule induces a geometric structure on the phase space, or, in other words, the microscopic dynamics are confined to group orbits of the Lie group associated with the Lie algebra.
This perspective offers three distinct coordinate changes, yielding equivalent representations:

\begin{enumerate}[label=(\arabic*)]
    \item \textbf{Change of basis within the algebra:}
    A change of basis within the VGL algebra maps the linear combination of generators and associated coefficients $a_\alpha(t)$ of the microscopic equations to a different one --- while leaving the vector field invariant. This affects the form of the microscopic and macroscopic equations and the associated matrix representations, while Lie bracket structure, the group, the superposition rule, and the geometry of the foliation remain preserved. - A familiar example is provided by the quasilinear ensemble (see Sec.~\ref{subsec:quasilinear_examples} and Ref.~\cite{augustsson2025quasilinear}): we are free to choose fundamental solutions up to a $\text{GL}(\ds)$ change of basis. To visualize, for an initial value problem with $\ds=2$, the $2\dn$ integration constants $c_{i,\ell}$ are defined up to a $\text{GL}(2)$ change of basis, depending on the choice of fundamental solutions; in particular,  prescribing initial conditions for the macroscopic variables $(U_\ell, V_\ell)$ eliminates four parameters, leaving $\gamma = 2n - 4$ quantities that label the leaves.

    \item \textbf{Isomorphisms between algebra representations:} An isomorphism between different representations of the same Lie algebra changes the \textit{form} of the vector field, while preserving the Lie structure that guarantees the reduction. This operation amounts to a state space isomorphism and is distinct from a mere change of basis: rather than merely relabeling or combining generators within one fixed representation, the algebra is mapped to a different set of vector fields, which may result in distinct microscopic and macroscopic ODE structures. - An example is given with the Bernoulli ensemble, see Sec.~\ref{subsec:bernoulli_examples}: the isomorphism  $\phi(\Lg_D)=\Lg_1, \phi (\Lg_T)=-1/(p-1)\Lg_2)$ maps $\aff{1}$ with generators $x\partial_x$ and $\partial_x$ to the algebra spanned by $x \partial_x$ and $x^p\partial_x$, resulting in the nonlinear Bernoulli equation from a linear affine structure. The nonlinear transform $y_k=x^{1-p}$ linearizing the Bernoulli equation is the state-space analog of this algebraic isomorphism.

    \item \textbf{Diffeomorphisms within the group orbit:} Diffeomorphisms within the group orbit re-parameterize the macroscopic variables: this changes only the form of the macroscopic equations, while the algebraic representation, the microscopic dynamics given by $\dot\xx_i=f(\xx_i,\mfmf)$  and the invariant manifold remain preserved  --- the microscopic states are recovered via the (transformed) superposition $\xx=\Sf(\M_1,\M_2,\ldots,\M_m;\bm{\Xi})$. Thus, this type of transformation is fundamentally different from (1) and (2). - An example is the conversion between macroscopic variables for Riccati ensembles, i.e., between canonical coordinates $(q_1,q_2,q_3)$ and Cestnik-Martens coordinates $(Q,y,s)$, where only $Q$ satisfies a Riccati equation (see Sec.~\ref{subsec:riccati_examples}).
\end{enumerate}

Readers familiar with Hamiltonian mechanics may relate these transformations to canonical transformations.
\smallskip
\paragraph{Open questions and outlook.}
Several open questions remain. We are currently pursuing the extension to heterogeneous nodal dynamics~\cite{OttAntonsen2008} and the breakdown of superposition at non-generic points---where fundamental solutions collide and the cross-ratio diverges. Moreover, the constructive nature of the framework---testing any candidate ODE for VGL closure, or, constructing an ODE with desired symmetries---turns the discovery of new reducible systems into a systematic procedure. This may extend the family of network dynamical systems amenable to dimensional reductions.
Furthermore, it is noteworthy that the finite-dimensional VGL algebras for vector fields on one- and two-dimensional manifolds have been locally classified~\cite{lopezkamranolverlie1992}, thus allowing to assess the existence of a superposition rule; however, in higher dimensions a complete classification remains unavailable; for a comprehensive review of Lie groups and VGL algebras, we refer the reader to Ref.~\cite{carinena2011lie}. Finally, it would be interesting to explore whether a certain concept of approximate symmetries may broaden the applicability of our framework.

\section*{Acknowledgments}
We thank R. Mirollo and J. Engelbrecht for earlier discussions that helped stimulate our interest in the subject, and C. Kuehn for helpful insights and discussions.

\section*{Authorship}
Both authors contributed equally in all parts of the research and preparation of the manuscript.
% \smallskip\\
Generative AI was used for language editing and to explore existing mathematical concepts during the development of this manuscript. Mathematical results, computations, and interpretations were independently derived and verified by the authors.

\end{document}